\documentclass[letterpaper,11pt]{amsart}

\usepackage[all]{xy}                        %

\CompileMatrices                            

\UseTips                                    

\input xypic
\usepackage[bookmarks=true]{hyperref}       

\usepackage{amssymb,latexsym,amsmath,amscd}
\usepackage{xspace}

\usepackage{graphicx}

\reversemarginpar

\theoremstyle{plain}
\newtheorem{theorem}{Theorem}[section]
\newtheorem*{theorem*}{Theorem}
\newtheorem{proposition}[theorem]{Proposition}
\newtheorem{corollary}[theorem]{Corollary}
\newtheorem{lemma}[theorem]{Lemma}

\theoremstyle{definition}
\newtheorem{definition}[theorem]{Definition}

\newtheorem{remark}[theorem]{Remark}

\newcommand{\enm}[1]{\ensuremath{#1}}          %
\newcommand{\op}[1]{\operatorname{#1}}
\newcommand{\cal}[1]{\mathcal{#1}}

\newcommand{\ZZ}{\enm{\mathbb{Z}}}

\newcommand{\PP}{\enm{\mathbb{P}}}

\newcommand{\Aa}{\enm{\cal{A}}}
\newcommand{\Bb}{\enm{\cal{B}}}

\newcommand{\Ee}{\enm{\cal{E}}}
\newcommand{\Ff}{\enm{\cal{F}}}
\newcommand{\Gg}{\enm{\cal{G}}}
\newcommand{\Hh}{\enm{\cal{H}}}
\newcommand{\Ii}{\enm{\cal{I}}}

\newcommand{\Kk}{\enm{\cal{K}}}

\newcommand{\Mm}{\enm{\cal{M}}}
\newcommand{\Nn}{\enm{\cal{N}}}
\newcommand{\Oo}{\enm{\cal{O}}}

\newcommand{\Ss}{\enm{\cal{S}}}

\newcommand{\Vv}{\enm{\cal{V}}}
\newcommand{\Ww}{\enm{\cal{W}}}

\renewcommand{\phi}{\varphi}
\renewcommand{\theta}{\vartheta}
\renewcommand{\epsilon}{\varepsilon}

\newcommand{\Hom}{\op{Hom}}
\newcommand{\Ext}{\op{Ext}}

\renewcommand{\to}[1][]{\xrightarrow{\ #1\ }}

\newcommand{\old}[1]{}

\begin{document}

\title[Globally generated vector bundles]{On higher rank globally generated vector bundles over a smooth quadric threefold}
\author{E. Ballico, S. Huh and F. Malaspina}
\address{Universit\`a di Trento, 38123 Povo (TN), Italy}
\email{edoardo.ballico@unitn.it}
\address{Department of Mathematics, Sungkyunkwan University \\
Cheoncheon-dong, Jangan-gu \\
Suwon 440-746, Korea}
\email{sukmoonh@skku.edu}
\address{Politecnico di Torino, Corso Duca degli Abruzzi 24, 10129 Torino, Italy}
\email{francesco.malaspina@polito.it}
\thanks{The second author is supported by Basic Science Research Program 2012-0002904 through NRF funded by MEST}

\begin{abstract}
We give a complete classification of globally generated vector bundles of rank 3 on a smooth quadric threefold with $c_1\leq 2$ and extend the result to arbitrary higher rank case. We also investigate the existence of globally generated indecomposable vector bundles, and give the sufficient and necessary conditions on numeric data of vector bundles for indecomposability.
\end{abstract}

\maketitle
\section{Introduction}
Globally generated vector bundles on projective varieties play an important role in algebraic geometry. If they are non-trivial they must have strictly positive first Chern class.  Globally generated vector bundles on projective spaces with low first Chern class have been investigated in several papers. If $c_1(\Ee)=1$ then it is easy to see that modulo trivial summands we have only $\Oo_{\PP^n}(1)$ and $T\PP^n(-1)$.  The classification of rank $r$ globally generated vector bundles with $c_1=2$ is settled in \cite{SU}. In \cite{huh} the second author carried out the case of rank two  with $c_1=3$ on $\mathbb P^3$ and in \cite{ce} the authors continued the study until $c_1\leq 5$. This classification was extended to any rank in \cite{m} and to any $\mathbb P^n$ ($n\geq 3$) in \cite{am} and \cite{SU2}.  In \cite{e} are shown the possible Chern classes of rank two globally generated vector bundles
on  $\mathbb P^2$.

Let $Q$ be a smooth quadric threefold over an algebraically closed field of characteristic zero. In our previous  paper \cite{BHM} we investigated the existence of globally generated vector bundles of rank 2 on $Q$ with $c_1\leq 3$. The aim of this paper is to study the existence of globally generated indecomposable vector bundles of higher rank with $c_1\leq 2$.  First we have a complete classification of globally generated vector bundles of rank 3 on $Q$ with the first Chern class 1 or 2 :
\begin{theorem}
Let $\Ee$ be a non-split vector bundle of rank 3 on $Q$ with $c_1\leq 2$. $\Ee$ is globally generated if and only if $\Ee$ admits an exact sequence,
$$0\to \Oo_Q^{\oplus 2} \to \Ee \to \Ii_C(c_1) \to 0,$$
where $C$ is a smooth irreducible curve of degree $d$ and genus $g$:
\begin{align*}
c_1=1 \Rightarrow & ~~(d,g)=(1,0)~;~ \Ee \simeq \Sigma \oplus \Oo_Q\\
                      &~~(d,g)=(2,0)~;~ \Ee\simeq \Aa_P\\
c_1=2 \Rightarrow &~~(d,g)=(3,0)~;~ \Ee\simeq \Sigma \oplus \Oo_Q(1)\\
&~~ (d,g)\in \{(4,0), (4,1), (5,1), (6,2), (8, 5)\} , \text{or}
\end{align*}
$\Ee$ is isomorphic to $\Aa_P^{\vee}(1)$ or a pull-back of $\Nn_{\PP^3}(1)\oplus \Oo_{\PP^3}$ whose associated curve $C$ is a disjoint union of two conics.
\end{theorem}

Here, $\Sigma$ is the spinor bundle of $Q$, $\Aa_P$ is the pull-back of $T\PP^3(-1)$ along the linear projection $Q \to \PP^3$ with a center $P\in \PP^4\setminus Q$ and $\Nn_{\PP^3}$ is a null-correlation bundle on $\PP^3$. The vector bundle $\Ee$ has the Chern classes $(c_2,c_3)$ with $c_2=\deg (C)$ and $c_3=2g-2+d(3-c_1)$.

Secondly, we extend the theorem to vector bundles of arbitrary higher rank and investigate their indecomposability.

\begin{theorem}
There exists a globally generated and indecomposable vector bundle of rank $r\geq 3$ on $Q$ with the Chern classes $(c_1, c_2, c_3)$, $c_1\leq 2$, if and only if the numeric data $(c_1, c_2, c_3;r)$ is one of the followings:
\begin{align*}
&(1,2,2;3\leq r\leq 4),\\
&(2,4,0;3)~,~ (2,4,2;3)~,~ (2,4,4;4),\\
&(2,5,5;4\leq r\leq 5)~,~ (2,6,8;4\leq r\leq 7)~,~(2,8,16;4\leq r\leq 13) .
\end{align*}
\end{theorem}
To show the existence of indecomposable vector bundles on $Q$ with $c_1=2$, we consider a family $\mathfrak{F}$ of indecomposable vector bundles with $c_1=1$ and construct extensions of elements of $\mathfrak{F}$ by themselves. Then we check the indecomposability case by case.

As an automatic consequence, every globally generated vector bundle of rank $r\geq 14$ on $Q$ with $c_1\leq 2$ is decomposable.

The work was initiated during the stay of the second author at the Politecnico di Torino and the second author would like to thank the institute, especially the third author for warm hospitality.

\section{Preliminaries}
 Let $Q$ be a smooth quadric hypersurface in $\PP^4$ and let $\Ee$ be a coherent sheaf of rank $r$ on $Q$. Then we have:
\begin{align*}
c_1(\Ee(k))&=c_1+kr\\
c_2(\Ee(k))&=c_2+2k(r-1)c_1+2k^2{r\choose 2}\\
c_3(\Ee(k))&=c_3+k(r-2)c_2+2k^2{r-1\choose 2}+2k^3{r\choose 3} \\
\chi(\Ee)~~~~~~&=(2c_1^3-3c_1c_2+3c_3)/6+3(c_1^2-c_2)/2+13c_1/6+r,
\end{align*}
where $(c_1, c_2, c_3)$ is the Chern classes of $\Ee$. In particular, when $\Ee$ is a vector bundle of rank 2 with $c_1=-1$, we have
$$\chi(\Ee)=1-c_2~~,~~ \chi(\Ee(1))=6-2c_2~~,~~ \chi(\Ee(-1))=0,$$
$$\chi(\mathcal{E}nd(\Ee))=7-6c_2.$$

\begin{proposition}\cite{sierra}\label{prop1}
Let $\Ee$ be a globally generated vector bundle of rank $r$ on $Q$ such that $H^0(\Ee(-c_1))\not= 0$, where $c_1$ is the first Chern class of $\Ee$. Then we have
$$\Ee\simeq \Oo_Q^{\oplus r-1}\oplus \Oo_Q(c_1).$$
\end{proposition}
In particular, $\Ee \simeq \Oo_Q^{\oplus r}$ is the unique globally generated vector bundle of rank $r$ on $Q$ with $c_1=0$. Thus let us assume that $c_1\geq 1$.

Let $\Ee$ be a globally generated vector bundle of rank $r\geq 3$ on $Q$. If $c_3(\Ee)=0$, then by Lemmas 4.3.1 and 4.3.2 in \cite{OSS},  we have the sequence:
\begin{equation}\label{eq1}
0\rightarrow \Oo_Q^{\oplus(r-2)} \rightarrow \Ee \rightarrow \Ff \rightarrow 0,
\end{equation}
where $\Ff$ is a vector bundle of rank 2 on $Q$ with $c_i(\Ff)=c_i(\Ee)$ for all $i$. Since $\Ff$ is globally generated, we can use the result on the classification of globally generated vector bundles of rank 2 with $c_1\leq 3$ \cite{BHM}. In this classification, we have $h^1(\Ff^{\vee})=0$ and so the sequence (\ref{eq1}) splits, i.e. $\Ee\simeq \Oo_Q^{\oplus (r-2)}\oplus \Ff$.

Let $\Ee$ be a globally generated vector bundle of rank $3$ on $Q$ with the first Chern class $c_1=c_1(\Ee)$ and then it fits into the following exact sequence,
\begin{equation}\label{eqa2}
0\rightarrow \Oo_Q^{\oplus 2} \rightarrow \Ee \rightarrow \Ii_C(c_1) \to 0,
\end{equation}
where $C$ is a smooth curve of degree $c_2(\Ee)$ on $Q$ \cite{BC}\cite{Chang}. If $C$ is empty, then $\Ee$ is isomorphic to $\Oo_Q^{\oplus 2}\oplus \Oo_Q(c_1)$ and so let us assume that $C$ is not empty. It would mean that we assume that $H^0(\Ee(-c_1))=0$.

\begin{remark}\label{pa}
As in the proof of the theorem 4.1 \cite{Hartshorne1}, the third Chern class of $\Ee$ fitted into the sequence (\ref{eqa2}) is $2p_a(C)-2+d(3-c_1)$ where $p_a$, $d$ are the arithmetic genus and degree of $C$. In particular, for a normal rational curve of degree $4$ with $c_1=2$, the corresponding $\Ee$ has the Chern classes $(c_1, c_2, c_3)=(2,4,2)$.
\end{remark}

Let us first deal with the case of $c_1=1$.

\begin{definition}
Let $\phi_P : Q \to \PP^3$ be the linear projection with the center point $P\in \PP^4\setminus Q$. We define
$$\Aa_P := \phi_P^*  (T\PP^3(-1)),$$
the pull-back of $T\PP^3(-1)$, the tangent bundle of $\PP^3$ twisted by $-1$, along $\phi_P$. Without confusion, we simply denote $\Aa_P$ by $\Aa$.
\end{definition}

In the proof of Theorem 2.1 in \cite{OS}, $\Aa_P$ is uniquely determined by the choice of $P$.
\begin{remark}
$\Aa$ admits an exact sequence
\begin{equation}\label{ees1}
0\to \Oo_Q(-1) \to \Oo_Q^{\oplus 4} \to \Aa \to 0.
\end{equation}
From the long exact sequence of symmetric powers associated to the sequence (\ref{ees1}), twisted by $\Oo_Q(-1)$, we obtain
$$H^0(\wedge^q \Aa(-1))=0~~, ~~\text{for }q\in 1,2.$$
It implies the stability of $\Aa$ due to the Hoppe criterion.
\end{remark}

\begin{proposition}
Let $\Ee$ be a globally generated vector bundle of rank 3 on $Q$ with $c_1=1$. Then $\Ee$ is isomorphic to either $$\Oo_Q^{\oplus 2} \oplus \Oo_Q(1)~,~ \Aa_P ~~ ( P\in \PP^4\setminus Q )~,\text{  or  } ~\Oo_Q\oplus \Sigma.$$
\end{proposition}
\begin{proof}
 Since $\Ii_C(1)$ is globally generated from the sequence (\ref{eqa2}), $C$ is contained in a complete intersection $X$ of two hyperplane sections of $Q$. In particular, the degree of $C$ is at most 2 and so $C$ is either a conic or a line. If $C$ is a conic, i.e. $C=X$, then from a locally free resolution of $\Ii_C$, we obtain a locally free resolution of $\Ee$ :
$$0\to \Oo_Q(-1) \to \Oo_Q^{\oplus 4} \to \Ee \to 0.$$
Notice that $\Ee$ is isomorphic to $\Aa$, a pull-back of the tangent bundle of $\PP^3$ twisted by $-1$ (Theorem 2.1 in \cite{OS}). If $C$ is a line $l$, then it is a zero locus of a section of the spinor bundle $\Sigma$ which gives us the following exact sequence:
$$0\to \Sigma(-1) \to \Oo_Q^{\oplus 3} \to \Ii_l (1)\to 0.$$

From the following diagram we can deduce that $\Ee\simeq \Oo_Q\oplus \Sigma$.
$$\begin{array}{ccccccc}
& & &0& &0 & \\
& & &\downarrow & &\downarrow &\\
&          & &\Sigma(-1)   &  =  &\Sigma(-1)  & \\
&& &\downarrow & &\downarrow & \\
0 \to &\Oo_Q^{\oplus 2}&\rightarrow & \Oo_Q^{\oplus 5} &\rightarrow & \Oo_Q^{\oplus 3}& \to 0\\
&\|& &\downarrow & &\downarrow &\\
0 \to &\Oo_Q^{\oplus 2} & \rightarrow & \Ee &\rightarrow & \Ii_l(1) &  \to 0\\
&& &\downarrow & & \downarrow& \\

& &          &0& & 0& \\
\end{array}$$
\end{proof}

\section{Case of rank 3 and $c_1=2$}

Now assume that $\Ee$ is a globally generated vector bundle of rank 3 on $Q$ with $c_1(\Ee)=2$, then we have the sequence (\ref{eqa2}) with $c_1=2$.
In particular, $\Ii_C(2)$ is globally generated.

\begin{lemma}\cite{Arrondo} \label{arr}
No line is a connected component of $C$. In particular, the degree of $C$ is at least 2.
\end{lemma}
\begin{proof}
Let $l$ be a line component of $C$. If we tensor the sequence (\ref{eqa2}) with $\Oo_l$,  then $\bigwedge^2 N_{l, Q}\otimes \Oo_l(-2)\simeq \Oo_l(-1)$ is globally generated, which is absurd.
\end{proof}

As the first case let us assume that $H^0(\Ee(-1))\not= 0$ and then $C$ is contained in a complete intersection $X$ of two hypersurface sections of degree $1$ and $2$, so the degree of $C$ is at most $4$. By Lemma \ref{arr}, $C$ has no line as its connected component.
\begin{proposition}\label{prop2}
If $\Ee$ is a globally generated vector bundle of rank 3 on $Q$ with $c_1=2$ and $H^0(\Ee(-1))\not= 0$, then $\Ee$ is one of the followings:
\begin{enumerate}
\item $\Oo_Q^{\oplus 2} \oplus \Oo_Q(2)$
\item $\Oo_Q\oplus \Oo_Q(1)^{\oplus 2}$
\item $\Sigma\oplus \Oo_Q(1)$
\item $0\to \Oo_Q(-1) \to \Oo_Q^{\oplus 3} \oplus \Oo_Q(1) \to \Ee \to 0$.
\end{enumerate}
\end{proposition}
\begin{proof}
By Lemma \ref{arr}, the degree of $C$ is at least $2$. If $\deg (C)=2$, then $C$ is a conic, a complete intersection of two hyperplane sections of $Q$. Since $h^1(\Oo _Q(t))=0$ for $t=-1,0$, the equations
of the generators of $\Ii_C$ lifts to sections of $H^0(\Ee (-t+2))$. So we have an exact sequence
\begin{equation}\label{eq++}
0\to \Oo_Q \stackrel{\psi}{\to} \Oo_Q^{\oplus 2} \oplus \Oo_Q(1)^{\oplus 2} \to \Ee \to 0,
\end{equation}
where $\psi = (\psi _1,\psi _2)$ with $\psi _1: \Oo _Q\to \Oo_Q^{\oplus 2} $ and $\psi _2: \Oo _Q\to \Oo_Q(1)^{\oplus 2} $. The map $\psi _1$ is given by two constants. Hence
if $\psi _1\ne 0$, then the sequence (\ref{eq++}) splits. If $\psi _1=0$, i.e. $\psi =(0,\psi _2)$ then the common zero-locus of two linear forms is non-empty and so $\Ee$ is not locally free, a contradiction. Thus we have $\Ee \cong  \Oo_Q \oplus \Oo_Q(1)^{\oplus 2}$.

If $\deg (C)=3$, then $C$ is a twisted cubic. So there is a a surjective map $\alpha:\Oo_Q(-2)^{\oplus 2}\oplus\Oo_Q(-1)\to \Ii_C$ and let $\Kk=\ker(\alpha)$ :
\begin{equation}\label{kernel}
0\to \Kk \to \Oo_Q(-2)^{\oplus 2} \oplus \Oo_Q(-1) \stackrel{\alpha}{\to} \Ii_C \to 0.
\end{equation}
The Chern classes of $\Ii_C(2)$ is $(c_1, c_2, c_3)=(2,\deg (C), 2p_a(C)-2+3)=(2,3,1)$ by Remark \ref{pa}. From the sequence (\ref{kernel}), the third Chern class of $\Kk(2)$ is 0. In particular, $\Kk$ is locally free. Moreover the map $\alpha(t)$, the twist of $\alpha$ by $\Oo_Q(t)$, is surjective on global section for any integer $t$, so $\Kk$ is ACM. Now since $c_1(\Kk(2))=-1$ and $c_2(\Kk(2))=1$, we get $\Kk(2)=\Sigma(-1)$. Then we have the following exact sequence
$$0\to\Sigma(-1)\to\Oo_Q^{\oplus 2}\oplus\Oo_Q(1)\to \Ii_C(2)\to 0.$$
Using the sequence (\ref{eqa2}), since $\Ext^1(\Oo_Q^{\oplus 2}\oplus\Oo_Q(1), \Oo_Q^{\oplus 2})=0$, we get the sequence:
$$0\to \Sigma(-1) \to \Oo_Q^{\oplus 4} \oplus \Oo_Q(1) \to \Ee \to 0.$$
Let us consider the dual sequence of it. Since the spinor bundle is globally generated by the $4$ sections, $\Ee^\vee$ is ACM
and so we can deduce that $\Ee \simeq \Sigma\oplus \Oo_Q(1)$.

If $\deg (C)=4$, then $C$ is equal to the complete intersection $X$, i.e. $C$ is a smooth curve of type $(2,2)$ in a smooth quadric surface with
$$0\to \Oo_Q(-1) \to \Oo_Q\oplus \Oo_Q(1) \stackrel{\beta}{\to} \Ii_C(2) \to 0.$$
Since $\Ext^1 (\Oo_Q\oplus \Oo_Q(1), \Oo_Q^{\oplus 2})=0$, we obtain the surjectivity of the map
$$\Hom (\Oo_Q \oplus \Oo_Q(1), \Ee) \to \Hom (\Oo_Q\oplus \Oo_Q(1), \Ii_C(2)).$$
Thus there exists a map $\beta' : \Oo_Q\oplus \Oo_Q(1) \to \Ee$ inducing $\beta$, with which we obtain the resolution (4) in the assertion with $(\alpha, \beta) : \Oo_Q^{\oplus 3}\oplus \Oo_Q(1) \to \Ee$ where $\alpha : \Oo_Q^{\oplus 2} \to \Ee$ is the map in the sequence (\ref{eqa2}).
\end{proof}

\begin{remark}
Conversely, if $C$ is a smooth elliptic curve of degree 4, it is ACM and hence $h^1(\Ee(t))=0$ for all $t\in \ZZ$. Since $\Ii_C(2)$ is globally generated, so is $\Ee$. In this case, we have $(c_2,c_3)=(4,4)$.
\end{remark}

Now assume that $H^0(\Ee(-1))=0$. Note that $C$ is contained in a complete intersection of two hypersurface sections of degree 2 and so $c_2=\deg (C)\leq 8$. Since $C$ is not contained in a hyperplane section with Lemma \ref{arr}, we have $4\leq c_2 \leq 8$.

\begin{proposition}
Let $\Ee$ be a globally generated vector bundle of rank 3 on $Q$ with $(c_1, c_2)=(2,4)$ and $h^0(\Ee(-1))=0$. Then $\Ee$ is isomorphic to one of the followings :
\begin{enumerate}
\item $\Aa_P^{\vee}(1)$ for some $P \in \PP^4 \setminus Q$,
\item a pull-back of $\Nn_{\PP^3}(1)\oplus \Oo_{\PP^3}$ or
\item a quotient of $\Sigma \oplus \Sigma$ by $\Oo_Q$.
\end{enumerate}
\end{proposition}

\begin{proof}
 Let us assume that $C$ is not connected. Then $C$ is a disjoint union of two conics $C_i$. Let $P_i$ be the projective plane containing $C_i$, then the intersection point $P:=P_1\cap P_2$ is in $\PP^4\setminus Q$. The projection from the point $P$ defines a double covering $l_P : Q \to \PP^3$ sending $C_i$ to a line $L_i$ in $\PP^3$. Since $\Ii_{L_1\cup L_2}(2)$ is globally generated, so is $\Ii_C(2)$. The fact that $H^1(\Oo_Q)=0$ implies that the vector bundle $\Ee$ fitted into the sequence (\ref{eqa2}) is always globally generated. On $\PP^3$, $\Ext^1 (\Ii_{L_1\cup L_2}(2), \Oo_{\PP^3}^{\oplus 2})$ is isomorphic to $H^0(\Oo_{L_1\cup L_2})^{\oplus 2}$. By a pull-back to $Q$, it maps isomorphically to $\Ext^1 (\Ii_C(2), \Oo_Q^{\oplus 2})\simeq H^0(\Oo_C)^{\oplus 2}$. Note that the only vector bundles on $\PP^3$ that is given by the non-trivial extension of $\Ext^1 (\Ii_{L_1\cup L_2}(2), \Oo_{\PP^3}^{\oplus 2})$ are $\Omega_{\PP^3}(2)$ and $\Nn_{\PP^3}(1)\oplus \Oo_{\PP^3}$ \cite{SU}. Thus $\Ee$ is a pull-back of them, i.e. $\Ee\simeq \Aa^{\vee}(1)$ or a pull-back of $\Nn_{\PP^3}(1)\oplus \Oo_{\PP^3}$. In this case, we have $(c_2,c_3)=(4,0)$.

Now assume that $C$ is connected; hence either $C$ is a smooth elliptic curve linearly normal
in a hyperplane of $\mathbb {P}^4$ or $C$ is a rational normal curve. Since $H^0(\Ee(-1))=0$, so $C$ is a normal rational curve of degree 4. Since $h^1(\Ee^{\vee})=h^2(\Ii_C(-1))-2=1$, so there exists a unique non-trivial extension $\Hh$ of $\Ee$ by $\Oo_Q$,
\begin{equation}\label{p1}
0\to \Oo_Q \to \Hh \to \Ee \to 0.
\end{equation}
 By the Riemann-Roch Theorem, we obtain that for any $a\geq 0$,
$$  h^2(\Ii_C(a))=h^1(\Oo_C(a))=-4a-1=0.$$
From the dual of the sequence (\ref{p1}), we obtain that for any $t<0$,
$$h^1(\Hh^\vee(t))=h^1(\Ee^\vee(t))=h^2(\Ee(-t-3))=h^2(\Ii_C(-t-1))=0$$
and $H^1(\Hh^\vee)=0$. Since $C$ is projectively normal, we have $H^1_*(\Ii_C)=0$ and so $H^1_*(\Ee)=H^2_*(\Ee^\vee)=0$. It implies that $H^2_*(\Hh^\vee)=0$.
  Let us notice that $H^3(\Hh^\vee(-2))\cong H^3(\Ee^\vee(-2))\cong H^0(\Ee(-1))=0$.
  Thus we have $H^i(\Hh^\vee(1)\otimes\Oo_Q(-i))=0$ for $i>0$ and by the `Castelnuovo-Mumford criterion', $\Hh^\vee(1)$ is regular. Then $\Hh(t)$ is also regular for any $t\geq 1$ and in particular  $H^1_*(\Hh^\vee)=H^2_*(\Hh^\vee)=0$. Since $\Hh$ is ACM with $(c_1,c_2,c_3)=(2,4,2)$, so we have $\Hh\cong\Sigma\oplus\Sigma$.
 In this case, we have $(c_2,c_3)=(4,2)$ for $\Ee$.
 \end{proof}
Here null-correlation bundles $\Nn_{\PP^3}$ are cokernels of maps $T_{\PP^3}(-1) \to \Oo_{\PP^3}(1)$ and so they form a family of projectively equivalent but not isomorphic vector bundles. A pull-back of $\Nn_{\PP^3}(1) \oplus \Oo_{\PP^3}$ depends on the choice of the center $P\in \PP^4\setminus Q$ of the projection $\PP^4 \dashrightarrow \PP^3$.

\begin{remark}
Conversely to the last case of the previous proposition, we obtain the followings :

\begin{enumerate}
\item Any normal rational curve $C$ of degree 4 is ACM. Thus we have $h^1(\Ee(t))=0$ for all $t\in \ZZ$ and $\Ii_C(2)$ is globally generated. Since the line bundle $\omega_C^{\vee}(1)$ is globally generated, so is $\Ee$.
\item Conversely, let $s, s'$ be two sections of $\Sigma$ whose zeros are two disjoint lines $l,l'$ respectively. It gives us an exact sequence:
$$0\to \Oo_Q \stackrel{(s,s')}{\to} \Sigma^{\oplus 2} \to \Ee \to 0,$$
where $\Ee$ is a vector bundle of rank 3 on $Q$ with $(c_1,c_2)=(2,4)$ since the section $(s,s')$ does not vanish. Note that $h^0(\Ee)=7$ and $h^0(\Ee(-1))=0$. So the corresponding curve $C$ to $\Ee$ is a normal rational curve of degree 4.
\end{enumerate}
\end{remark}
From the proof of the previous Proposition, we can obtain the following statement :
\begin{corollary}\label{cor}
$\Sigma \oplus \Sigma$ is the unique vector bundle of rank $4$ on $Q$ with $(c_1,c_2,c_3)=(2,4,2)$ and with no trivial factor.
\end{corollary}

Now let us deal with the case $c_2\geq 5$.
\begin{proposition}\cite{f}
Let $t(d,g)$ be the number of trisecant lines of a smooth and connected curve $C\subset \PP^4$ of degree $d$ and genus $g$ in $\PP^4$. Then we have
$$t(d,g)=\frac{(d-2)(d-3)(d-4)}6 -g(d-4).$$
\end{proposition}
\begin{remark}
When $t(d,g) \ne 0$, there is at least one
line $L\subset \mathbb {P}^4$ such that $\deg (L\cap C)\ge 3$ . If $t(d,g)<0$, then there
are infinitely many lines as above. Hence if $t(d,g) \ne 0$, then $\mathcal {I}_C(2)$ is not spanned. Note that $t(5,0) = 1$, $t(6,0) = 4$, $t(6,1) = 2$ and $t(7,3) = 1$.
\end{remark}

\begin{lemma}\label{conic}
If $\Ee$ is globally generated with $c_2\geq 5$, then the associated curve $C$ has no smooth conic as its connected component.
\end{lemma}
\begin{proof}
Let $C=E_1\sqcup T$ with $\deg (T)\ge 3$, $T$ not necessarily connected. Let $Q'$ be a general hyperplane section containing $E_1$. Then $Q'$ is a smooth quadric surface, $E_1$ has type $(1,1)$ and $T\cap Q'$ is a zero-dimensional scheme; $\mathcal {I}_C(2)$
is not globally generated because $\mathcal {I}_{E_1\cup (Q'\cap T),Q'}(2,2) \cong \mathcal {I}_{T\cap Q',Q'}(1,1)$
is not globally generated.
\end{proof}

\begin{lemma}
Let $\Ee$ be a globally generated vector bundle of rank 3 on $Q$ with $c_1=2$.
We have $c_2=5$ or $6$ if and only if the associated curve $C$ is smooth and irreducible with $\deg(C)=c_2$ and $p_a(C)=c_2-4$.
\end{lemma}
\begin{proof}
Let us assume that $c_2=5$. By Lemmas \ref{arr} and \ref{conic}, the associated curve $C$ is connected. First assume that $C$ is contained in a hyperplane $H$. Set $Q':= H\cap Q$. If $Q'$ is a quadric cone, then $C$ contains its vertex (Exercise V.2.9 in \cite{Hartshorne}).
Hence even if $Q'$ is not smooth the curve $C$ has infinitely many 3-secant lines.
Hence $\mathcal {I}_C(2)$ is not spanned. Hence $C$ spans $\mathbb {P}^4$. Hence
$C$ has genus $g\in \{0,1\}$. Since $t(5,0)=1$, we have $g=1$. In this case $C$ is ACM and $\mathcal {I}_C(2)$
is spanned. In this case, we have $(c_2,c_3)=(5,5)$.

Let us assume that $c_2=6$ and the associated curve $C$ is not connected. Again by Lemmas \ref{arr} and \ref{conic}, we have $C=D_1\sqcup D_2$ with $D_i$
rational normal curve of a hyperplane $H_i$. Set $Q':= H_1\cap Q$. We cannot have $H_2=H_1$,
because $D_1\cap D_2 = \emptyset$. Hence $\deg (D_2\cap Q')=3$; enough to have
a line $L\subset Q'$ with $\deg (L\cap (D_1\cup (Q'\cap D_2)) \ge 3$; hence $\mathcal {I}_C(2)$
is not globally generated. Now assume that $C$ is connected. If $C$ is contained in a hyperplane, then
$C$ is not contained in two other quadrics, because it has degree $>4$, a contradiction.
By Castelnuovo's upper bound for the genus, $Y$ has genus $g\in \{0,1,2\}$. Since $t(6,0)=4$ and $t(6,1)=2$, so we have $g=2$.
$\mathcal {I}_{C,\mathbb {P}^4}(2)$ is globally generated and so is $\mathcal {I}_C(2)$. In this case, we have $(c_2,c_3)=(6,8)$.
\end{proof}

\begin{remark}\label{i0}
Let $C\subset \mathbb {P}^4$ be a linearly normal curve of genus $g\in \{0,1,2\}$ and degree
$g+4$. Notice that $g+4 \ge 2g+2$ in all cases. Hence the homogeneous ideal of $C$
in $\mathbb {P}^4$ is generated by quadrics (old result, first proved in \cite{f}, see \cite{gl}, p. 302, for a proof and the history
of the theorem). Hence $\Ii_{C,\mathbb {P}^4}(2)$ is spanned. Hence $\Ii _C(2)$ is spanned. Hence
these cases give spanned vector bundles. The same is true for $c_2=4$ when we take as $C$
a linearly normal elliptic curve of a hyperplane of $\mathbb {P}^4$, because this curve
is a complete intersection inside $Q$ of a hyperplane and a degree $2$ hypersurface.
\end{remark}
\begin{lemma}
There is no globally generated vector bundle of rank 3 on $Q$ with $(c_1,c_2)=(2,7)$.
\end{lemma}
\begin{proof}
Let us assume that the associated curve $C$ is not connected. Since we excluded the case of components of degree $\le 2$, we have $C = D_1\sqcup D_2$ with $D_1$ a rational normal curve of degree 3 and $\deg (D_2)=4$; this case is excluded as in the case $c_2=6$. Now let us assume that $C$ is connected. Since $C$ is linked by a complete intersection $X$ of two quadric hypersurfaces to a line $L$, we have the following exact sequence
$$0\to \Ii_X(2) \to \Ii_C(2) \to \omega_L(1) \to 0,$$
which is not possible since $\omega_L(1)\simeq \Oo_L(-1)$ is not globally generated.
\end{proof}

As the final case, let us assume that $c_2=8$. Since $C$ is a a complete intersection, we have a surjection
$\Oo _Q^{\oplus 2} \to \Ii _C(2)$. This surjection and the sequence (\ref{eqa2}) give :
\begin{equation}\label{eqc1}
0\to \Oo_Q(-2) \stackrel{\psi}{\to} \Oo_Q^{\oplus 4} \stackrel{\phi}{\to} \Ee \to 0.
\end{equation}
Notice that $\psi$ is given by $4$ sections of $H^0(\mathcal {O}_{Q}(2))$ without any
common zero in $Q$. Since $\Ee$ has no trivial direct summand, so $h^0(\Ee^{\vee})=0$ and $h^1(\Ee^{\vee})=10$. On the other hand, we have $h^1(\Ii_C)=h^1(\Ee(-2))=0$ and so $h^0(\Oo_C)=1$. In particular, $C$ is an irreducible and smooth curve of degree $8$ and genus $5$. Conversely any such $4$ sections give an exact sequence
(\ref{eqc1}) in which $\Ee$ is a globally generated vector bundle with rank $3$ and $(c_1,c_2,c_3)=(2,8,16)$. Any such $4$ sections are linearly independent. A general $4$-dimensional linear
subspace of $H^0(\mathcal {O}_{Q}(2))$ has no common zero in $Q$, but some of them
have a common zero. Hence the set of all exact sequences (\ref{eqc1}) is parametrized
(not one-to-one) by a non-empty open subset $\mathcal {B}$ of the Grassmannian
of all $3$-dimensional linear subspaces of $\mathbb {P}^{13}$. In particular the set of all
such bundles $\Ee$ is irreducible and unirational. We do not know when two points of $\mathcal {B}$ gives isomorphic bundles
or $\mbox{Aut}(Q)$-isomorphic bundles (i.e. bundles $\Ee$, $\Ee'$ such that there
is $f\in \mbox{Aut}(Q)$ with $\Ee'\cong f^\ast (\Ee)$).

Note that $h^1(\mathcal {I}_C(t))=0$ for all $t$.
Hence $h^1(\Ee(t)) =0$ for all $t$. Since $h^0(\mathcal {I}_C(1))=0$, we have
$h^0(\Ee(-1)) =0$. Hence $\Ee$ is stable if and only if $h^0(\Ee^\vee (1)) =0$, i.e. if and
only if $h^3(\Ee(-2)) =0$. Since $h^0(\mathcal {O}_{Q}(-2))=0$ and $h^4(\mathcal {\Ff})=0$
for each coherent sheaf $\mathcal {\Ff}$ on $Q$, we get that any such $\Ee$ is stable.

Form (\ref{eqc1}) we get $h^0(\Ee)=4$. Hence the map $\phi$ in (\ref{eqc1}) is uniquely
determined by $\Ee$ and a choice of a basis of $H^0(\Ee) \cong \mathbb {K}^4$.
Hence $\mathcal {B}$ is a one-to-one parametrization of such bundles.

Summarizing the arguments so far, we have the following:
\begin{proposition}
A vector bundle $\Ee$ of rank 3 on $Q$ with $c_1=2$ and $H^0(\Ee(-1))=0$ is globally generated if and only if
$$(c_2,c_3)\in \{(4,0), (4,2), (5,5), (6,8), (8,16) \}.$$
Except when $(c_2,c_3)=(4,0)$, the associated curve $C$ is a smooth irreducible curve of degree $c_2$ and genus $(c_3-c_2+2)/2$.
\end{proposition}

\section{Higher Rank Case}
Let $\Ee$ be a globally generated vector bundle of rank $r>3$ on $Q$. We know that it fits into an exact sequence
\begin{equation}\label{eqa7}
0 \to \Oo _Q^{\oplus(r-3)} \to \Ee \to \Ff \to 0
\end{equation}
where $\Ff$ is a globally generated vector bundle of rank $3$ on $Q$ with $c_i(\Ff)=c_i(\Ee )$, $i=1,2,3$.
Conversely, since $h^1(\Oo _Q)=0$, if $\Ff$ is a rank $3$ spanned vector bundle and $\Ee$ is any coherent sheaf fitting into the sequence (\ref{eqa7}), then $\Ee$ is a rank $r$ spanned vector bundle with
$c_i(\Ee)=c_i(\Ff)$, $i=1,2,3$, and $h^0(\Ee )=h^0(\Ff)+r-3$. This does not give us a complete classification, but only a very rough one unless $h^1(\Ff^\vee )=0$ (e.g. if $\Ff$ splits or is isomorphic to a direct sum of a line bundle and a twist of the spinor bundle; in these cases the sequence (\ref{eqa7}) splits and hence $\Ee \cong \Oo _Q^{\oplus (r-3)}\oplus \Ff$.)

We assume $c_1\in \{1,2\}$ and set $\alpha (\Ff):= h^1(\Ff^\vee ) = h^2(\Ff(-3))$. From the exact sequence
\begin{equation}\label{eqd2}
0 \to \Oo _Q^{\oplus 2}\to \Ff \to \mathcal {I}_C(c_1) \to 0
\end{equation}
where $C$ is a smooth curve of degree $d:= c_2$ and genus $g$, we have
$h^1(\Oo _C(c_1-3)) = (c_1-3)d+g-1$ by the Riemann-Roch theorem. Since $h^3(\Oo _Q(-3))=1$, we get
\begin{equation}\label{eqd3}
(3-c_1)d+g-3 \le \alpha (\Ff) \le (3-c_1)d+g-1
\end{equation}

\begin{lemma}\label{uuu}
For the existence of a trivial direct summand for $\Ee$, we obtain the following statements :
\begin{enumerate}
\item Let $\Ee$ be a vector bundle fitting in (\ref{eqa7}) with the extension induced
by $e_1,\dots ,e_{r-3} \in H^1(\Ff^\vee )$. Then $\Ee$ has $\Oo _Q$ as a direct factor if and
only if $e_1,\dots ,e_{r-3}$ are not linearly independent.

\item If $\alpha(\Ff)<r-3$, then any $\Ee$ fitting in (\ref{eqa7}) has $\Oo _Q^{\oplus (r-3-\alpha (\Ff))}$ as a direct
factor. If $0 < r-3 \le \alpha (F)$, then $\Ee $ given by a general extension (\ref{eqa7}) has no factor isomorphic to $\Oo _Q$.
\end{enumerate}
\end{lemma}

\begin{proof}
Let us take $\Ee$ given by (\ref{eqa7}) with respect to the extension classes
$e_1,\dots ,e_{r-3} $ in $H^1(\Ff^\vee )$. First assume that $e_1,\dots ,e_{r-3}$ are linearly dependent. Changing a basis of the trivial bundle  $\Oo _Q^{\oplus (r-3)}$ we reduce to the case $e_{r-3}=0$. In this
case $\Ee \cong \Gg \oplus \Oo _Q$ with either $\Gg = \Ff$ (case $r=4$) or $\Gg$ extension of $\Ff$
by $r-4$ copies of $\Oo _Q$ using the extensions $e_1,\dots ,e_{r-4}$ (case $r\ge 5$).
Now assume that $\Oo _Q$ is a direct factor of $\Ee$ and write $\Ee = \Oo _G\oplus \Mm$.
Any map $\Oo _Q \to \Oo _Q$ is either an isomorphism or the zero map.
First assume that the composition $\psi$ of the map $ \Oo _Q^{\oplus (r-3)} \to \Ee$ with the projection
$\Oo _Q$ is non-zero. We get that $\psi$ is surjective. Linear algebra says
that we may apply an endomorphism of $ \Oo _Q^{\oplus (r-3)}$ after which the first factor
of $ \Oo _Q^{\oplus (r-3)}$ goes isomorphically onto the first factor of $\mathcal {O}_Q\oplus \Gg$.
Call $f_1,\dots ,f_{r-3}\in H^1(\Ff^\vee )$ the extensions in the new basis. The extensions
$f_2,\dots ,f_{r-3}$ give $\Gg$, while $f_1=0$. Hence $e_1,\dots ,e_{r-3}$ are not linearly independent.

Part (ii) follows from part (i).
\end{proof}

\begin{lemma}\label{gg0}
We have $h^1(\mathcal {A}^\vee ) =1$.
\end{lemma}
\begin{proof}
It is derived from the previous Lemma with $\Ee=\Oo_Q^{\oplus 4}$ and $\Ff=\Aa$.
\end{proof}

Hence there is a non-split extension
\begin{equation}\label{eqgg0}
0 \to \Oo _Q \to \mathcal {B}_{\mathcal {A}} \to \mathcal {A} \to 0
\end{equation}
Since $h^1(\mathcal {A}^\vee ) =1$ there is, up to isomorphisms, a unique vector bundle $\mathcal {B}_{\mathcal {A}}$ fitting in an extension (\ref{eqgg0}) for a fixed
bundle $\mathcal {A}$.

\begin{definition}
For the tangent bundle $T\PP^4$ of $\PP^4$, we define
$$\Phi:= T\PP^4(-1)\vert _Q.$$
\end{definition}
By its definition, $\Phi$ admits the resolution :
\begin{equation}\label{eqaa0}
0 \to \Oo _Q(-1) \to \Oo _Q^{\oplus 5} \to \Phi \to 0
\end{equation}

\begin{lemma}\label{ggg1}
For each $\mathcal {A}$, we have $\mathcal {B}_{\mathcal {A}} \cong \Phi$.
\end{lemma}

\begin{proof}
The vector bundle $T\PP ^4(-1)\vert _Q$ is spanned and hence it fits into an exact sequence
\begin{equation}\label{eqgg00}
0 \to \Oo_Q \to T\PP ^4(-1)\vert _Q \to \Ff \to 0
\end{equation}
for some spanned rank $3$ vector bundle $\Ff$. Since $c_3( T\PP ^4(-1)\vert _Q) \ne 0$, the classification given in the rank $3$ case gives
$ \Ff \cong \mathcal {A}$ for some pull-back bundle $\mathcal {A}$. Since $h^1(\Omega _{\PP^4}(-1)) =0$, the exact sequence
$$0 \to \Omega _{\PP^4}(-1) \to \Omega _{\PP^4}(1)\to \Phi  ^\vee \to 0$$gives $h^0(\Phi ^\vee )=0$. Hence $\Phi$ has no trivial factor.
Hence (\ref{eqgg00}) does not split. Hence $T\PP ^4(-1)\vert _Q \cong \mathcal {B}_{\mathcal {A}}$ for some $\mathcal {A}$. For any two pullbacks
$\mathcal {A}$ and $\mathcal {A}'$ there is $g\in \mbox{Aut}(Q)$ such that $g^\ast (\mathcal {A}) \cong \mathcal {A}'$, because $\mbox{Aut}(Q)$
acts transitively on the set of all points of $\mathbb {P}^4\setminus Q$ and any pull-back twisted tangent bundle is uniquely determined
by its center of projection. Hence $g^\ast (\mathcal {B}_{\mathcal {A}}) \cong \mathcal {B}_{\mathcal {A}'}$.
\end{proof}

\begin{lemma}\label{ba0}
The bundle $\Phi$  is simple and we have
$$h^1(\Phi^\vee )=h^1(\Phi ^\vee \otimes \Phi )=0.$$
\end{lemma}

\begin{proof}
Note that $T\PP^4(-1)$ is simple and $\Phi=T\PP^4(-1)|_Q$. Since we have $h^1(\Omega _{\PP ^4}(-1)) =h^2(\Omega _{\PP^4}(-2))=0$, so $h^1(T\PP^4 \otimes \Omega_{\PP^4}(-2))=0$ by the Euler sequence. It implies that $\Phi$ is simple.

Using the exact sequence
$$0 \to \Omega _{\PP^4}(-1) \to \Omega _{\PP^4}(1) \to \Omega _{\PP^4}(1)\vert _Q\to 0,$$
we obtain $h^1(\Omega_{\PP^4}(1)|_Q)=0$ and so $h^1(\Phi^{\vee})=0$.

For the last, let us twist (\ref{eqaa0}) by $\Phi $ to get the exact sequence
\begin{equation}\label{eqaa3}
0 \to \Phi ^\vee \otimes \Phi \to \Phi^{\oplus 5} \to \Phi (1)\to 0.
\end{equation}
From the sequence (\ref{eqaa0}), we get $h^1(\Phi )=0$ and so $h^0(\Phi^{\vee}\otimes \Phi)=1$ since $\Phi$ is simple. Again from the sequence (\ref{eqaa0}), we have $h^0(\Phi )=4$ and $h^0(\Phi (1))=19$. Now we can use (\ref{eqaa3}) to obtain $h^1(\Phi^{\vee}\otimes \Phi)=0$.
\end{proof}

There is no more non-trivial extension of $\Phi$ by $\Oo_Q$ due to the previous lemma.

Since $h^1(\Sigma (1)^\vee )=0$, Lemmas \ref{gg0} and \ref{ggg1} give the following result.
\begin{proposition}\label{5.6}
$\Ee$ is a globally generated vector bundle of rank $r$ on $Q$ with $c_1(\Ee )=1$ if and only if we have
$$\Ee \simeq \Oo_Q^{\oplus (r-k)} \oplus \Ff,$$
where $\Ff \in \{\Oo_Q(1), \Sigma , \Aa,  \Phi\} $ and $k=\mathrm{rank} (\Ff)$.

\end{proposition}

Now let us assume that $\Ee$ is a globally generated vector bundle of rank $r\ge 4$ on $Q$ with $c_1(\Ee )=2$. For any $\Ee$ and $\Ff$ fitting into the sequence (\ref{eqa7}),
$\Ee$ is spanned if and only if $\Ff$ is spanned since $h^1(\Oo_Q)=0$. Since $h^1(\Oo _Q(-1))=0$, we have $h^0(\Ee (-1)) = h^0(\Ff (-1))$.

The globally generated vector bundle $\Ff$ of rank $3$ with $c_1(\Ff)=2$ and $h^0(\Ff (-1)) >0$ is classified in Proposition \ref{prop2}. Thus we can describe the possible $\Ee$ with $H^0(\Ee(-1))\not=0$.

When $\Ff$ is isomorphic to one of the first three vector bundles in Proposition \ref{prop2}, i.e. either $\Oo_Q^{\oplus 2} \oplus \Oo_Q(2)$, $\Oo_Q\oplus \Oo_Q(1)^{\oplus 2}$ or $\Oo_Q(1)\oplus \Sigma$, then we have $h^1(\Ff^{\vee})=0$ and thus the sequence (\ref{eqa7}) splits.

\begin{proposition}\label{c2}
Let $\Ee$ be a globally generated vector bundle of rank $r>3$ on $Q$ with $c_1=2$ such that $H^0(\Ee(-1))\not= 0$. Then we have
$$\Ee\simeq \Oo_Q^{\oplus (r-k)}\oplus \Ff,$$
 where $\Ff$ is a vector bundle of rank $k$ in the set
$$\{ \Oo_Q(2), \Oo_Q(1)^{\oplus 2} , \Oo_Q(1)\oplus \Sigma, \Oo_Q(1)\oplus \Aa, \Oo_Q(1)\oplus \Phi, \Ee_P\}.$$
Here, $\Ee_P$ is a vector bundle of rank $4$ uniquely determined by a point $P\in Q$ in the sequence (\ref{Ep}).
 \end{proposition}

\begin{proof}
It is enough to consider the case when $\Ff$ is associated to a smooth elliptic curve $C$, the complete intersection of a hyperplane and a hypersurface of degree $2$ in $Q$. So it fits into an exact sequence
\begin{equation}\label{eqeeee1}
0\to \Oo_Q(-1) \to \Oo_Q^{\oplus 3}\oplus \Oo_Q(1) \to \Ff \to 0.
\end{equation}
Since $\Ff$ is globally generated and without trivial factors, we have
$h^3(\Ff (-3)) = h^0(\Ff^\vee )=0$. From (\ref{eqeeee1}) we get $h^1(\Ff ^\vee ) = h^2(\Ff(-3))
= 5-3$. Thus if $r\ge 6$ we see that $\Ee$ has $\Oo _Q^{\oplus (r-5)}$ as its direct factor. So it is enough to deal with the cases $r=4,5$. Since $h^0(\Ff ) =8$ and so we have $h^0(\Ee ) =r+5$.

\quad (a) Let us take any decomposable vector bundle $\Ee$ that fits into the sequence (\ref{eqa7}), if there is any. Since $h^0(\Ff (-1))
=1$, there is a decomposition $\Ee \cong \Vv\oplus \Ww$ with
$\Vv$ indecomposable and $h^0(\Ww(-1)) =0$. The classification of all spanned vector bundles with $c_1=1$
gives $\Vv \cong \Oo _Q(1)$. Hence $\Ww$ is a spanned vector bundle
of rank $r-1$ with $c_1(\Ww)=1$, $h^0(\Ww(-1))=0$ and $h^0(\Ww) = r$. Proposition \ref{5.6} gives
that either $\Ww\cong \mathcal {A}\oplus \Oo _Q^{\oplus (r-4)}$ or $r\ge 5$ and $\Ww\cong \Phi \oplus \Oo _Q^{\oplus (r-5)}$. We also know that $\mathcal {A}\oplus \Oo _Q(1)$
and $\Phi \oplus \Oo _Q(1)$ give $\Ff$ associated to smooth elliptic curves of degree 4.

\quad (b) Assume $r=4$. Since $h^0(\Ee (-1))>0$ and $h^0(\Ee (-2)) =0$, there is an exact sequence
\begin{equation}\label{eqgg1}
0 \to \Oo _Q(1) \to \Ee \to \mathcal {G} \to 0
\end{equation}
with $\mathcal {G}$ a rank $3$ spanned torsion free sheaf with $c_1(\mathcal {G})=1$ and $h^0(\mathcal {G}) =4$. Since $h^0(\mathcal {G})=4$ and
$\mathcal {G}$ is spanned, we have an exact sequence
\begin{equation}\label{eqgg2}
0 \to \mathcal {L} \stackrel{\psi}{\to} \Oo _Q^{\oplus 4} \to \mathcal {G} \to 0
\end{equation}
Since $\mathcal {G}$ is torsion free, $\mathcal {L}$ is reflexive (\cite{Hartshorne1}, Proposition 1.1). Since $Q$ is a smooth threefold and $\mathcal {L}$ has rank $1$, $\mathcal {L}$
is a line bundle (\cite{Hartshorne1}, Proposition 1.9). Since $h^0(\Oo _Q^{\oplus 4})=h^0( \mathcal {G})$, we have $\mathcal {L} \cong \Oo _Q(e)$ with $e<0$. Since $\mathcal {G}$ is torsion free, it is locally free outside a curve. Taking the Segre classes in (\ref{eqgg1})
and (\ref{eqgg2}) (or, equivalently, restricting to a general line) we get $e=-1$.  Hence $\psi $ is given by $4$ linear forms $L_1,\dots ,L_4\in H^0(\Oo _Q(1))$ with $L_i\ne 0$ for some
$i$.  First assume that the forms $L_1,\dots ,L_4$ are not linearly independent. Up to an automorphism of $\Oo _Q^{\oplus 4}$ we may assume $L_4=0$. From (\ref{eqgg2})
we get that $\Oo _Q$ is a direct factor of $\mathcal {G}$. Since $\Ee$ is globally generated, we get that $\Oo_Q$ is a direct factor of $\Ee$. Hence $\Ee \cong \Oo _Q\oplus \Ff$ with
$\Ff$ one of the bundles described in Proposition \ref{prop2}.

Now assume that $L_1,\dots ,L_4$ are linearly independent. Hence the zero-locus of these forms in $\mathbb {P}^4$ is a single point, $P$. First assume
$P\in Q$. Notice that $L_1,\dots ,L_4$ are uniquely determined by $P$, up to an automorphism of $\Oo _Q^{\oplus 4}$.
Hence for a given $P$, there is a unique sheaf $\mathcal {G}$ and we call it $\mathcal {G}_P$. This sheaf is locally free outside $P$, while
not locally free at $P$, but it has homological dimension $2$, because it fits in an exact sequence
\begin{equation}\label{eqgg3}
0\to \Oo _Q(-1)\to \Oo _Q^{\oplus 4} \to \mathcal {G}_P\to 0
\end{equation}

From (\ref{eqgg3}) we get that $h^2(\Gg _P(-2)) =1$. Take a general hyperplane section $Q'$ of $Q$. In particular we can assume $P\notin Q'$. From (\ref{eqgg3}) we get
an exact sequence on $Q'$:
$$0 \to \Oo _{Q'}(-1) \to \Oo _{Q'}^{\oplus 4} \to \Gg _P\vert_{ Q'}\to 0$$
Taking duals we
get
$$0 \to (\Gg _P\vert _{Q'})^\vee (t) \to \Oo _{Q'}(t)^{\oplus 4} \to \Oo _{Q'}(t+1)\to 0$$
Since $L_1,\dots ,L_4$ generate the homogeneous ideal of $P$ in $\PP^4$, we get the surjectivity
of the map
$$H^0(Q', \Oo _{Q'}(t)^{\oplus 4}) \to H^0(Q',\Oo _{Q'}(t+1))$$
for every $t\ge 0$. Hence
$h^1(Q',(\Gg _P\vert_{ Q'})^\vee (t))=0$ for every $t\ge 0$ and it implies that $h^1(Q',\Gg _P\vert _{Q'} (t)) =0$ for all $t \le -2$ by the Serre duality. From the exact sequence
$$0 \to \Gg_P(t-1) \to \Gg _P(t) \to \Gg _P(t)\vert _{Q'}\to 0$$
we get that the sequence $\{h^2(\Gg _P(t))\}_{t}$ is non-decreasing for $t\le -2$. By Remark 2.5.1 in \cite{Hartshorne1},
for $s\gg 0$, we have $h^2(\Gg _P(-s)) = h^0(\mathcal{E} xt^1(\Gg _P,\omega _Q)) >0$
(since $\Gg _P$ is not locally free). Hence $h^2(\Gg _P(t))=1$ for any $t\leq -2$. In particular $h^2(\Gg _P(-4))=\dim \Ext^1(\Gg_P,\Oo_Q(1))=1$.
Hence for a fixed $P$ there is, up to isomorphism,
a unique sheaf fitting in a non-trivial extension (\ref{Ep}). This sheaf is locally free and we call it $\Ee_P$.
\begin{equation}\label{Ep}
0\to \Oo_Q(1) \to \Ee_P \to \Gg_P \to 0.
\end{equation}

 Let $L\subset Q$ be a line not containing $P$. Since $\mathcal {G}_P\vert _L$ is spanned, it has splitting type $(1,0,0)$. Hence (\ref{eqgg1}) gives that $\Ee _P\vert _L$ has splitting type $(1,1,0,0)$. Let $D\subset Q$
be a line containing $P$. Since $\mathcal {G}_P$ is not locally free at $P$ and $D$ is a smooth curve, $\mathcal {G}_P\vert _D$ is a direct sum of a rank $3$ vector bundle
$A$ and a torsion sheaf $\tau$ supported by $P$. Since $\mathcal {G}_P$ is not locally free at $P$, the vector space $\mathcal {G}_P\vert _{\{P\}}$ has dimension
$\ge 4$. Hence $\tau \ne \emptyset$. Restrict (\ref{eqgg1}) to $D$ we get that the inclusion $\Oo_Q(1)\vert _D \cong \Oo _D(1) \to \Ee _P\vert_D$ has cokernel with torsion.
Since $\Ee_P\vert _L$ is spanned and with degree $2$, we get $\Ee _P\vert _D$ has splitting type $(2,0,0,0)$. Hence $\Ee _P$ and $\Ee _O$ are not isomorphic
if $O\ne P$, but $g^\ast (\Ee _P) \cong \Ee _O$ for any $g\in \mbox{Aut}(Q)$ such that $g(P)=O$.

Now assume $P\notin Q$. Hence $\mathcal {G}$ is a locally free sheaf with rank $3$ spanned by $4$ sections. The universal property
of $T\PP^3(-1)$ gives the existence of a morphism $\phi :Q\to \PP^3$ such that $\mathcal {G} \cong \phi ^\ast (T\PP^3(-1))$. Since
$c_1(\mathcal {G})=1$, we get $\mathcal {G} \cong \mathcal {A}$. In Lemma \ref{a9}, we will prove that $h^1(\Aa^{\vee}(1))=0$ and it gives $\Ee \cong \mathcal {A}\oplus \Oo _Q(1)$.

By Lemma \ref{gg6}, the vector bundles $\Ee_P$ give only $ \Oo_Q^{\oplus (r-4)}\oplus \Ee_P$ for $r\ge 5$.
\end{proof}

\begin{lemma}\label{gg6}
For any $P\in Q$ we have $h^1(\Ee _P^\vee )=0$
\end{lemma}

\begin{proof}
Serre duality gives $h^1(\Ee _P^\vee )=h^2(\Ee _P(-1))$. Since $\mathcal {L} \cong \Oo _Q(-1)$, (\ref{eqgg2}) gives $h^2(\mathcal {G}_P(-1))=0$.
Hence (\ref{eqgg1}) gives $h^2(\Ee _P(-1))=0$.
\end{proof}

\begin{remark}
From (\ref{eqgg2}) we get $h^1(\Gg _P(t))=0$ for all $t\in \mathbb {Z}$. Hence (\ref{eqgg1}) gives $h^1(\Ee _P(t)) =0$ for all $t\in \mathbb {Z}$. We
also proved that $\Ee _P$ is not uniform: for each line $L\subset Q$ with $P\in L$ (resp. $P\notin L$) the bundle $\Ee _P\vert _L$ has splitting type
$(2,0,0,0)$ (resp. $(1,1,0,0)$). Hence in the list of Proposition \ref{c2} the bundles  $\Oo_Q^{\oplus (r-5)}\oplus \Ee_P$, $P\in Q$, are the only
non-uniform ones. \end{remark}
\section{Indecomposability}
Now notice from the computation of $c_3$ for globally generated vector bundles of rank 3 on $Q$ with $c_1=2$, $c_2\geq 4$, we  have the following statement:

\begin{proposition}\label{c3}
There exists no globally generated vector bundle of rank $r>3$ on $Q$ with $c_1=2$ if either $c_2=7$ or $c_2\ge 9$. The pairs $(c_2, c_3)$ with which there exist globally generated vector bundles of rank $r>3$ on $Q$ with the Chern classes $(c_1=2, c_2\ge 4, c_3)$, are as follows:
$$\{ (4,0),(4,2),(4,4),(5,5),(6,8),(8,16)\}.$$
\end{proposition}

Let us assume that $H^0(\Ee(-1))=0$, from which we can exclude the case $(c_2,c_3)=(4,4)$. Since the vector bundle $\Ff$ in the sequence (\ref{eqa7}) fits into the sequence (\ref{eqa2}), so we have
$$h^1(\Ff^{\vee})=h^2(\Ff(-3))= h^2(\Ii_C(-1))-h^3(\Oo_Q(-3)^{\oplus 2})=h^2(\Ii_C(-1))-2.$$
From the sequence
$$0\to \Ii_C(-1) \to \Oo_Q(-1) \to \Oo_C(-1) \to 0,$$
we have $h^2(\Ii_C(-1))=h^1(\Oo_C(-1))=h^0(\Oo_C(1)\otimes \omega_C)$ and so
$$h^1(\Ff^{\vee})= \left\{
                                           \begin{array}{llll}
                                             0, & \hbox{if $C$ is a disjoint union of two conics;}\\
                                             1, & \hbox{if $C$ is a normal rational curve of degree $4$;} \\                                          3, & \hbox{if $C$ is an elliptic curve of degree $5$;}\\
                                             5, & \hbox{if $C$ is a curve of genus 2 and degree $6$;}\\
                                             10, & \hbox{if $C$ is a curve of genus 5 and degree $8$.}
                                           \end{array}
                                         \right.$$

Let us fix any indecomposable globally generated vector bundle $\Ff$ of rank 3 on $Q$ with $c_1(\Ff)=2$
and $c_2(\Ff) =c_2$. Let $\Ee$ be the general extension (\ref{eqa7}).
In each case we computed $h^1(\Ff^\vee ) = h^2(\Ff(-3))$ and it is not zero.
Hence $\Ee$ has no trivial factor.

Assume $\Ee \cong \Ee_1\oplus \Ee_2$ with neither $\Ee_1$ nor
$\Ee_2$ trivial. Since $\Ee$ is spanned, each $\Ee_i$ is spanned. Since $c_1(\Ee_1)+c_1(\Ee_2)=2$
and neither $\Ee_1$ nor $\Ee_2$ is trivial, we have $c_1(\Ee_1)=c_1(\Ee_2)=1$. Since $\Ee$ has
no trivial factor, neither $\Ee_1$ nor $\Ee_2$ has a trivial factor. Both $\Ee_1$ and $\Ee_2$ should
be obtained from Proposition \ref{5.6} in which we avoid the ones with trivial factors.
Thus $\Ee$ is isomorphic to one of the followings:

\begin{equation}\label{dec}
\left\{
\begin{array}{ll}
 \Oo_Q(1)\oplus\Aa, \Oo_Q(1)\oplus \Phi, \Sigma\oplus \Sigma , &\hbox{ when $c_2=4$ } \\
\Sigma\oplus\Aa, \Sigma\oplus \Phi , &\hbox{ when $c_2=5$}\\
\Aa_O \oplus\Aa_P, \Aa\oplus \Phi, \Phi \oplus \Phi , &\hbox{ when $c_2=6$}
 \end{array}
 \right.
 \end{equation}

Now we compute the dimensions of all possible cohomology groups
$$h^1(\mathcal{H}om (A,B)) \text{ with } A,B \in \mathfrak{F}=\{\Oo_Q(1),\Sigma, \mathcal {A}, \Phi \}.$$

The projection formula gives
\begin{align*}
& h^1(\mathcal {A}(t)) = h^1(T\PP^3(t-1)) + h^1(T\PP^3(t-2))\\
 &h^1(\mathcal {A}^\vee (t)) = h^1(\Omega _{\PP^3}(t+1)) +
h^1(\Omega _{\PP^3}(t)).
\end{align*}

\begin{lemma}\label{a9}
We have
\begin{enumerate}
\item $h^1(\mathcal {A}(-1)) = h^1(\mathcal {A}^\vee (1))=0$
\item $h^1(\Phi(-1))=h^1(\Phi^{\vee}(1))=0$.
\end{enumerate}
\end{lemma}

\begin{proof}
The projection formula gives $h^1(\mathcal {A}(-1)) = h^1(T\PP^3(-2)) +h^1(T\PP ^3(-2)) =0$ and $ h^1(\mathcal {A}^\vee (1))
= h^1(\Omega _{\PP^3}(2)) + h^1(\Omega _{\PP^3}(1))=0$.

For the second, we can apply the same idea in the proof of Lemma \ref{ba0}.
\end{proof}

\begin{lemma}
There is no globally generated and indecomposable vector bundle $\Ee$ of rank at least 4 on $Q$ with $(c_1,c_2)=(2,4)$ and $H^0(\Ee(-1))=0$.
\end{lemma}
\begin{proof}
Let $\Ee$ fit into the sequence (\ref{eqa7}) with $\Ff$ associated to the curve $C$, where $C$ is either a disjoint union of two conics or a normal rational curve of degree 4. In the first case, we have $h^1(\Ff^{\vee})=0$ and so the sequence splits. In particular, $\Ee$ has $\Oo_Q$ as its direct summand. In the second case, since we have $h^1(\Ff^\vee)=1$, the only possibility for the rank of indecomposable $\Ee$ is $4$ due to Lemma \ref{uuu}. By Corollary \ref{cor}, we have the unique vector bundle $\Sigma \oplus \Sigma$ of rank 4 with $(c_1,c_2,c_3)=(2,4,2)$ that is decomposable.
\end{proof}

Recall that the restriction map
$H^0(T\PP^4(-1)) \to H^0(\Phi)$ is bijective and that every non-zero section of $T\PP^4(-1)$ vanishes at a unique point
of $\PP^4$. Conversely, for each $P\in \PP^4$ there is a unique $s\in H^0(T\PP ^4(-1))\setminus \{0\}$. For any $P\in \PP^4$, let $s_P: \Oo _Q \to \Phi$ be the section of $\Phi$ vanishing at $P$. Since
$H^0(\Phi (-1))=0$, the sheaf $s_P(\Oo _Q)$ is a saturated subsheaf of $Q$. Set $\Bb _P:= \Phi /s_P(\Oo _Q)$. By construction $\Bb _P$ is a rank 3 torsion free sheaf
on $Q$ and has the same Chern classes as $\Phi$. Moreover, $\Bb _P$ is spanned and with no trivial factor. If $P\notin Q$, then $\Bb _P = \Aa _P$. If $P\in Q$,
then $\Bb _P$ is not locally free and $P$ is the unique point of $Q$ at which $\Bb _P$ is not locally free. Hence $\Bb _{P} \cong \Bb _{P'}$ if and only if $P=P'$.
Since $Q$ is homogeneous, all sheaves $\Bb _P$, $P\in Q$, are $\mbox{Aut}(Q)$-equivalent.

\begin{proposition}\label{indecom5}
Let $\Ee$ be a globally generated and indecomposable vector bundle of rank $r\geq 4$ on $Q$ with $(c_1,c_2,c_3)=(2,5,5)$. Then the rank $r$ of $\Ee$ is either 4 or 5 and in each case there exists such a vector bundle.
\end{proposition}
\begin{proof}
According to Lemma \ref{uuu} and $h^1(\Ff^{\vee})=3$, the possibility for the rank of $\Ee$ is either 4, 5 or 6. In the case of $r=4$, $\Ee$ is indecomposable since there is no rank 4 vector bundle with $c_2=5$ in the list (\ref{dec}). Since the case $r=6$ can be obtained from Proposition \ref{b1}, let us assume that $r=5$.

If $\Ee$ is a globally generated vector bundle of rank 5 on $Q$ with the prescribed Chern classes, then we have $h^1(\Ee^\vee )=1$. Hence Proposition \ref{b1} implies that $\Ee$ fits in an exact sequence
\begin{equation}\label{eqo+1}
0 \to \Oo _Q\to \Sigma \oplus \Phi \stackrel{\sigma}{\to} \Ee \to 0.
\end{equation}
From (\ref{eqo+1}) we get non-zero maps $u: \Sigma \to \Ee$ and $v: \Phi \to \Ee$. Since $h^0(\Ee (-1)) =0$ and $\Sigma$ is stable, $u$ is injective. Let
$\Theta$ be the saturation of $u(\Sigma )$ in $\Ee$. Since $\Ee /\Theta$ is a spanned torsion free sheaf with no trivial factors,
we have $c_1(\Theta )\le 1$. Since the sheaf $u(\Sigma)$ is reflexive and $c_1(u(\Sigma)) =1$, we get $\Theta = u(\Sigma )$.
Since $\sigma$ is surjective, the map $v$ induces a surjective map $v': \Phi \to \Ee /u(\Sigma)$. The sheaf $\mbox{ker}(v')$ is reflexive (\cite{Hartshorne1}, Proposition 1.1)
and with rank 1. Hence $\mbox{ker}(v') \cong \Oo _Q(c)$ for some $c\in \mathbb {Z}$ (\cite{Hartshorne1}, Proposition 1.9).
Hence there is $P\in \PP^4$ such that $\Ee /u(\Sigma ) \cong \Bb _P$. Since $\Phi$ and $\Sigma$ are stable and $\mu (\Sigma )>\mu (\Phi)$, we have
$h^0(\Sigma^\vee \otimes \Phi )=0$. Since $\Sigma$ is simple and $h^1(\Sigma^\vee )=0$, (\ref{eqo+1}) gives $h^0(\Sigma^\vee\otimes \Ee )=1$.
Hence $u$ is the unique non-zero map $\Sigma \to \Ee$. Hence the point $P$ is unique. In other words, there exists a unique $P\in \PP^4$ such that $\Ee$ fits into the following sequence:
$$0\to \Sigma \to \Ee \to \Bb_P \to 0.$$

Since $h^1(\Sigma )=0$ and $\Sigma$ and $\Bb _P$ are spanned, every extension of $\Bb _P$ by $\Sigma$ is spanned. Note that every decomposable $\Ee$ is isomorphic to $\Sigma \oplus \Aa _P$ for some $P\in \PP^4\setminus Q$.

Let us assume that $P\in Q$. From the spectral sequence of local
and global Ext-functors and its associated 5 terms long exact sequence (Ex. 2 at page 75 in \cite{mc} and Theorem 2.5 in \cite{Hartshorne1}), we get an exact sequence
\begin{equation}\label{eqo+2}
\mbox{Ext}^1(\Bb _P,\Sigma ) \stackrel{f}{\to} H^0(\mathcal{\Ee}xt^1(\Bb _P,\Sigma )) \to H^2(\Bb _P^{\vee } \otimes \Sigma ).
\end{equation}

\quad {\emph {Claim 1:}} $H^2(\Bb _P ^{\vee }\otimes \Sigma )=0$.

\quad {\emph {Proof of Claim 1:}} By the definition of $\Bb _P$ we have an exact sequence
\begin{equation}\label{eqo+3}
0 \to \Oo _Q \stackrel{\psi}{\to} \Phi \to \Bb _P\to 0.
\end{equation}
Dualizing (\ref{eqo+3}) and looking at the definition of the map $\psi$ we get the exact sequence
\begin{equation}\label{eqo+4}
0 \to \Bb _P^{\vee } \to \Phi ^{\vee }\to \Ii _P \to 0.
\end{equation}
Since $\Sigma$ is locally free, from (\ref{eqo+4}) we get the exact sequence
\begin{equation}\label{eqo+5}
0 \to \Sigma \otimes \Bb _P^\vee \to \Sigma \otimes \Phi ^{\vee } \to \Ii _P\otimes \Sigma \to 0.
\end{equation}
Since $\Sigma$ is spanned, we have $h^1(\Ii _P \otimes \Sigma )=0$. Serre duality
gives $h^2(\Sigma \otimes \Phi ^{\vee }) = h^1(\Sigma ^{\vee }\otimes \Phi (-3)) = h^1(\Sigma \otimes \Phi (-4))$. We tensor the exact sequence
$$0 \to \Oo _Q(-1) \to \Oo _Q^{\oplus 5} \to \Phi \to 0$$
with $\Sigma (-4)$. Since $h^1(\Sigma (-4)) =h^2(\Sigma (-5)) =0$, we get $h^1(\Sigma \otimes \Phi (-4))=0$, concluding the proof of {\it{Claim 1}}. \qed

By {\it{Claim 1}}, the map $f: \mbox{Ext}^1(\Bb _P,\Sigma ) \to H^0(\Ee xt^1(\Bb _P,\Sigma ))$ is surjective. Since $\Bb _P$ is locally free outside $P$, the sheaf
$\Ee xt^1(\Bb _P,\Sigma )$ has $P$ as its support and hence the integer $h^0(\Ee xt^1(\Bb _P,\Sigma ))$ only depends from the local behavior
of $\Sigma$ at $P$. Hence $h^0(\Ee xt^1(\Bb _P,\Sigma )) = 2\cdot h^0(\Ee xt^1(\Bb _P,\Oo_Q))$. Notice that (\ref{eqo+3}) implies
$h^0(\Ee xt^1(\Bb _P,\Oo_Q))>0$. Since $h^0(\Oo _Q/\Ii _P)=1$, from (\ref{eqo+4}) we get $h^0(\Ee xt^1(\Bb _P,\Oo_Q))=1$. Hence the middle term
of an extension $\epsilon$ of $\Bb _P$ by $\Sigma$ is locally free if $f(\epsilon )\ne 0$.
Since $h^0(\Ee xt^1(\Bb _P,\Sigma )) =2$, we get a 2-dimensional vector space of extensions $\Ee$ of $\Bb _P$ by $\Sigma$ with locally free middle term.
Since the map $\Sigma \to \Ee$ is uniquely determined, up to a constant, we get a 1-dimensional family of pairwise non-isomorphic extensions
of $\Bb _P$ by $\Sigma$. In other words, there is a 4-dimensional family of pairwise non-isomorphic globally generated and indecomposable vector bundles of rank 5 on $Q$ with prescribed Chern classes.
\end{proof}

\begin{remark}\label{5i}
Let us assume that $P\not \in Q$ and then the pull-back of the Euler sequence gives
$$0 \to \Sigma \otimes \mathcal {A}_P^\vee \to \Sigma ^{\oplus 4} \to \Sigma(1) \to 0.$$
Since $h^0(\Sigma^{\oplus 4} )=16$ and $h^0(\Sigma (1)) =16$, we have $h^1(\Sigma \otimes \mathcal {A}_P^\vee)=h^0(\Sigma \otimes \mathcal {A}_P^\vee)$. Recall that we have a sequence
$$0\to \Oo_Q^{\oplus 2} \to \Aa_P^{\vee}(1) \to \Ii_{C_1\sqcup C_2} (2) \to 0,$$
where $C_1$ and $C_2$ are two disjoint conics such that the projective planes containing each conics intersect at a single point. Then we have $h^0(\Sigma \otimes \Aa_P^\vee)=h^0(\Ii_{C_1\sqcup C_2} \otimes \Sigma(1))$. Assume the existence of $s\in H^0(\Ii _{C_1\sqcup C_2}\otimes \Sigma (1))$. If $s$
vanishes on a divisor, then it vanishes on a quadric surface $Q\cap H$ plus at most a line. No such scheme contains $C_1\sqcup C_2$.
Hence $s$ only vanishes in codimension $2$. Set $D:= (s)_0$. $D$ is a locally complete intersection scheme with degree $5$, arithmetic genus
$1$ and $\omega _D \cong \Oo_D$. Since $D\supseteq C_1\sqcup C_2$ and $\deg (D)=5$, we get $D = C_1\cup C_2\cup L$ with $L$ a line. To get $\omega _D\cong \Oo _D$
we should have $\omega _D\vert _{C_i} \cong \Oo _{C_i}$ and hence $\deg (C_i\cap L) =2$. We would get $\deg (L\cap (C_1\cup C_2))=4$
and there is no such a line in $Q$. Thus we have $h^1(\Sigma \otimes \Aa_P^\vee)=0$.

It also implies that $h^1(\Sigma \otimes \Phi^\vee)=0$ using the sequence (\ref{eqgg0}).
\end{remark}

In order to study spanned vector bundle on $Q$ with $(c_1,c_2,c_3;r)=(2,5,5;6)$ we need the following lemma:
\begin{lemma}\label{l1}The set $S$ of all smooth elliptic curves $C\subset Q$ such that $\deg (C)=5$ and $C$ is linearly normal in $\PP^4$ is irreducible. \end{lemma}
\begin{proof}
Fix $C\in \Ss$. Since $\deg (C)\ge 2p_a(C)+2$
and $C$ is linearly normal, we have $h^1(\PP ^4,\Ii _{C,\PP^4}(2)) =0$
and the homogeneous ideal of $C$ in $\PP^4$ is generated by quadrics \cite{f}. Hence $h^0(\PP^4,\Ii _{C,\PP^4}(2)) = 5$, $h^1(\Ii _C(2))=0$, $h^0(\Ii _C(2)) =4$ and $\Ii _C(2)$ is spanned.
We also know that $\Ii _{C,\PP^4}(2)$ is spanned.

Let $\Delta$ be the set of all $D\subset Q$ such that $D$ is a smooth complete intersection of $Q$, a quadric hypersurface and a hyperplane. For each $D\in \Delta$
we have $N_D \cong \Oo _D(2)\oplus \Oo _D(1)$ and so $h^1(D,N_D)=0$, $h^0(D,N_D) = 12$ and $h^1(D,N_D(-P))=0$ for all $P\in D$. Thus the set $\Delta$ is irreducible and of dimension 12, because it parametrizes
the smooth complete intersections curves in $Q$ of type $(2,1)$.

 Let $\Delta '$ be the set of all nodal curves $D\cup J \subset Q$ with $D\in \Delta$, $J$
a line and $\sharp (J\cap D)=1$. Each $D\cup J$ is connected, nodal and $p_a(D\cup J)=1$.

\quad {\emph {Claim 1:}} $\Delta '$ is irreducible and of dimension $14$.

\quad {\emph {Proof of Claim 1:}} For each $P\in Q$ the set $\eta (P)$ of all lines of $Q$ containing $P$ is isomorphic to $\mathbb {P}^1$. Fix $A\in \Delta $, $P\in A$
and $L\in \eta (P)$. If $A\cup L\notin \Delta '$, i.e. if $\deg (A\cap L) \ge 2$, then $L\subset Q_A$ where $Q_A$ is a hyperplane section of $Q$ and $Q_A\supset A$.
If $Q_A$ is a cone with vertex containing $P$. then $A\cup L\notin \Delta '$ for all $L\in \eta (P)$. If $Q_A$ is not a cone with vertex $P$, then
$A\cup L\in \Delta '$, except for one or two line $L\in \eta (P)$ (one line if $Q_A$ is singular, two lines if $Q_A$ is smooth). Now use the irreducibility of $\Delta$. \qed

\quad {\emph {Claim 2:}} Fix $E = A\cup L\in \Delta '$. Then the Hilbert scheme $\mbox{Hilb}(Q)$ of closed subschemes in $Q$ is smooth and 15-dimensional at $E$. Moreover $E$ is in the closure
$\overline{\mathcal {S}}$ of $\mathcal {S}$.

\quad {\emph {Proof of Claim 2:}} By \S 3 and \S 4 in \cite{hh}, we have $h^1(E,N_E) =0$. Hence $h^0(E,N_E) = 3\cdot \deg (E)=15$ and $\mbox{Hilb}(Q)$ is smooth of dimension $E$
at $Q$. By \cite{hh}, Theorem 4.1, $E$ is smoothable inside $Q$, i.e (by the definition of the set $\mathcal {S}$), we have $E\in \overline{\mathcal {S}}$. \qed

\quad {\emph {Claim 3:}} $\mathcal {S}$ is irreducible if for every irreducible component $\Gamma$ of $\mathcal {S}$ there is $E\in \Delta '$ such that $E\in \overline{\Gamma}$.

\quad {\emph {Proof of Claim 3:}}  Take two irreducible components $\Gamma _i$, $i=1,2$, of $\mathcal {S}$ and $E_i\in \Delta '$ such that $E_i\in \overline{\Gamma _i}$.
The set $\Delta '$ is irreducible from {\it{Claim 1}}. The Hilbert scheme $\mbox{Hilb}(Q)$ is smooth at each point of $\Delta '$. Hence $\overline{\Gamma }_1=
\overline{\Gamma }_2$. Hence $\Gamma _1=\Gamma _2$. \qed

\quad{\emph{Claim 4:}} Let $T\subset Q$ be the intersection of $Q$ with a general quadric hypersurface of $\PP^4$ containing $C$. Then $T$ is a smooth surface.

\quad{\emph{Proof of Claim 4:}} Since $\mathcal {I}_C(2)$ is spanned,
the scheme $T$ is a degree $4$ surface smooth outside $C$ (Bertini's theorem). Hence $T$ is irreducible. Fix $P\in C$. Let $A(C,P)$ denote the set
of all $W\in \vert \Ii _C(2)\vert$ which are singular at $P$. We have $\Ii _C/(\Ii _C)^2 \cong N_C^\vee$. Since $\Ii _C(2)$ is spanned, $N_C^\vee (2)$ is spanned.
Since $N_C^\vee (2)$ has rank $2$ and $N_C^\vee (2)$ is spanned, $A(C,P)$ is a linear subspace with codimension $2$ of $\vert \Ii _C(2)\vert$. Since
$\dim (C)=1$ and $T$ is general, we get $T\notin A(C,P)$ for all $P\in C$. Hence $T$ is smooth. \qed

\begin{proof}[Proof of Lemma \ref{l1}]
Fix $T$ as in {\it{Claim 4}}. Since $T\subset \PP^4$ is a smooth complete intersection of two quadric hypersurfaces, it is a smooth Del Pezzo surface of degree 4.
Hence $\mbox{Pic}(T) \cong \mathbb {Z}^6$ \cite{d}. Since $C$ is an integral curve of $T$, but not a line, $\deg (C) =5$ and $p_a(C)=1$, we may express $T$ as the blowing up of $\PP^2$ at $5$ points $P_1,\dots ,P_5$, so that, in the associated basis,
$C$ is represented by $(a,b_1,\dots ,b_5)\in \ZZ ^6$ with $a>0$, $a\ge b_1+b_2+b_3$, $b_1\ge b_2\ge b_3\ge b_4\ge b_5\ge 0$,
\begin{equation}\label{eqvaa1}
5 = 3a-b_1-b_2-b_3-b_4-b_5 \ ,
\end{equation}
\begin{equation}\label{eqvaa2}
a^2  = 5+\sum _{i=1}^{5} b_i^2.
\end{equation}
(see \cite{Hartshorne}, V, 4.12 and Ex. 4.8, or \cite{h1}, equation (2) at page 303, for the the case of a cubic surface). One solution is given by $a=3$, $b_1=b_2=b_3=b_4=1$
and $b_5=0$.

Fix another solution $(a,b_1,b_2,b_3,b_4,b_5)$. Since no plane curve of degree $\leq 2$ have normalization of genus 1 we may assume that $a\geq 3$. We may assume $b_1>0$, because 5 is not a multiple of 3.
Set $\epsilon := a-b_1-b_2-b_3$. We have $\epsilon \ge 0$ and $2a > b_1+b_2+b_3+b_4+b_5$. Hence it is sufficient to check all cases with $3 \le a \le 4$. Let us assume that $a=3$. From (\ref{eqvaa1}) and (\ref{eqvaa2}), we have $b_1+ \cdots + b_5=b_1^2+\cdots + b_5^2=4$ and so we have $b_1(b_1-1)+\cdots + b_5 (b_5-1)=0$. Since $n(n-1)\geq 0$ for all $n\in \ZZ$, each $b_i$ is either 0 or 1. Again by (\ref{eqvaa1}) we have $b_1=\cdots = b_4=1$ and $b_5=0$. When $a=4$, we similarly have $b_1(b_1-1)+\cdots + b_5(b_5-1)=-1$ and it is impossible.

Thus we can see $C$ as a subcurve of $T$ of type $(3,1,1,1,1,0)$. Take a smooth $D\in (3,1,1,1,1,1)$ and let $J$ be the line $(0,0,0,0,0,-1)$. Notice
that $D\in \Delta$ and $D\cup J\in \Delta '$. We may deform $C$ to $D\cup J$ inside $T$ in the linear system $\vert \Oo _T(3,1,1,1,1,0)\vert$. Hence we may deform $D\cup J$ inside $Q$, concluding the proof of lemma \ref{l1}.
\end{proof}

\begin{proposition}\label{b1}
Every spanned vector bundle on $Q$ with $(c_1,c_2,c_3;r)=(2,5,5;6)$ without trivial factors is isomorphic to $\Sigma \oplus \Phi$.
\end{proposition}

Let $\Ss$ be the set of all smooth elliptic curves $C\subset Q$ such that $\deg (C)=5$ and $C$ is linearly normal in $\PP^4$. We see $\Ss$ as an open subset of the Hilbert
scheme $\mbox{Hilb}(Q)$ of $Q$. Since
$h^0(C,\omega _C(1)) = 5$, it is sufficient to prove that for every $C\in \Ss$ we have $\Ee \simeq \Sigma \oplus \Phi$ in the exact sequence
$$0\to \Oo _Q^{ \oplus 5} \to \Ee \to \Ii _C(2)\to 0.$$
By Theorem 6.3. in \cite{AO} with $(t,j)=(0,1)$ and $\Ff = \Ee^\vee$, it is enough to prove that
$$h^1(\Ee^{\vee}(-2))=h^1(\Ee^{\vee})=h^2(\Ee^{\vee}\otimes \Sigma (-3))=0$$
so that $\Ee$ would have $h^1(\Ee^\vee (-1))=h^1(\Oo_C)=1$ factor of $\Phi$. Thus $\Ee$ would be $\Phi \oplus \Gg$ for some globally generated vector bundle $\Gg$ of rank 2 on $Q$ with $c_1=1$ and by Proposition 2.3 in \cite{BHM} $\Gg$ is isomorphic to either $\Sigma$ or $\Oo_Q\oplus \Oo_Q(1)$. Since $h^0(\Ee^\vee)=0$, so we would have $\Ee \simeq \Phi \oplus \Sigma$. Note that $h^1(\Ee^\vee (-2))=h^1(\Oo_C(1))=0$ and $h^1(\Ee^\vee)=h^2(\Ee(-3))$. Since $h^3(\Ee(-3))=h^0(\Ee^\vee)=0$, so we have $h^2(\Ee(-3))=h^2(\Ii_C(-1))-5=h^1(\Oo_C(-1))-5=h^0(\Oo_C(1))-5=5-5=0$.

 So it is sufficient to prove that $h^2(\Ee^\vee \otimes \Sigma (-3))=h^1(\Ii_C\otimes \Sigma (1))=0$ for every $C\in \Ss$. Set $\Ss ':= \{C\in \Ss :h^1(\Ii _C\otimes \Sigma (1)) =0\}$.
By the semicontinuity theorem for cohomology $\Ss '$ is an open subset of $\Ss$. Taking $5$ general sections of $\Sigma \oplus \Phi$ we get
$\Ss ' \ne \emptyset$. We want to show that $\Ss'=\Ss$. Let $\Ee$ be any bundle in in the closure $\overline{\Ss'}$ of $\Ss '$ in $\Ss$. By the semicontinuity
theorem for cohomology we have non-zero maps $u: \Sigma \to \Ee$ and $v: \Ee \to \Phi$ such
that $v\circ u = 0$. Since $\Sigma$ is stable, $h^0(\Ee (-1)) =0$ and $\mu (\Sigma ) > \mu (\Ee)$, we get that $u$ is injective. Let $\Theta$ be the saturation of $u(\Sigma )$
in $\Ee$. Since $\Ee$ is spanned and with no trivial factor, the torsion free
sheaf $\Ee /\Theta$ is spanned and with no trivial factor. Hence $c_1(\Ee /\Theta )>0$, i.e. $c_1(\Theta )\le 1$.
Since the sheaf $u(\Sigma )$ is reflexive and $c_1(u(\Sigma )) =1$, we get $\Theta = u(\Sigma )$. Since $v\circ u = 0$ and $u(\Sigma ) =\Theta$, $v$
induces a non-zero map $v': \Ee /u(\Sigma ) \to \Phi$ between torsion free sheaves with the same rank and $c_1=1$. Since $\Ee $ has no trivial factor,
every quotient of $\Ee /u(\Sigma )$ has positive $c_1$. Since $\Phi$ is stable, we get that $v'$ is injective. Since $c_i(\Ee /u(\Sigma )) = c_i(\Phi )$ for
all $i$, $v'$ is an isomorphism, i.e. $\Ee$ is an extension of $\Phi$ by $\Sigma $. Remark \ref{5i} gives
$\Ee \cong \Sigma \oplus \Phi$. Hence to by lemma \ref{l1}, we have $\Ss'=\Ss$, concluding the proof of Proposition \ref{b1}.
\end{proof}

\begin{lemma}\label{rm1}
For any vector bundles $\mathcal {A}_P$ and $\mathcal {A}_{O}$, we have
$$h^1(\Aa_O^\vee \otimes \Aa_P)=\left\{
                                                         \begin{array}{ll}
                                                         4, & \hbox{if $\Aa_O\cong \Aa_P$}; \\
                                                         3, & \hbox{if $\Aa_O \not\cong \Aa_P$}.
                                                         \end{array}
                                                      \right.$$
\end{lemma}
\begin{proof}
Let us consider the  exact sequence
\begin{equation}\label{eqrm1}
0 \to \mathcal {A}_O^\vee (-1) \to (\mathcal {A}_O^\vee )^{\oplus 4} \to \mathcal {A}_O^\vee \otimes \mathcal{A}_{P} \to 0
\end{equation}
and then we have
\begin{align*}
&h^2(\Aa_O^\vee (-1)) = h^1(\Aa _O(-2)) = h^1(T\PP ^3(-3)) + h^1(T\PP ^3(-4)) =0,\\
&h^1(\mathcal {A}_O^\vee(-1) )=h^1(\Omega _{\PP^3}) +h^1(\Omega _{\PP^3}(-1)) =1, \text{ and}\\
&h^1(\mathcal {A}_O^\vee )=h^1(\Omega _{\PP^3}(1)) +h^1(\Omega _{\PP^3}) =1.
\end{align*}
Then we can use (\ref{eqrm1}) to obtain the assertion.
\end{proof}

\begin{proposition}
There exists a globally generated and indecomposable vector bundle of rank $r\geq 4$ on $Q$ with $(c_1, c_2)=(2,6)$ if and only if we have $4\leq r \leq 7$.
\end{proposition}
\begin{proof}
As in the proof of Proposition (\ref{indecom5}), we have $4\leq r \leq 8$ where $r$ is the rank of a globally generated and indecomposable vector bundle $\Ee$. From the list (\ref{dec}) and Lemma (\ref{uuu}), a vector bundle in the general extension is indecomposable if $r$ is either 4 or 5. Let us assume that $r=6$. By Lemma \ref{rm1}, there exists a non-zero extension $\Aa_O$ by $\Aa_P$ :
\begin{equation}\label{eqrm2}
0 \to \Aa _P \to \Ee _{P,O} \to \Aa _O\to 0,
\end{equation}
where we allow the case $P=O$. Assume that $\Ee _{P,O}$ is decomposable. Since $h^1(\Aa _P)=0$, so $\Ee _{P,O}$ is globally generated. Hence the classification \ref{dec} gives $\Ee _{P,O} \cong \Aa _A\oplus \Aa _B$ for some $A, B\in \PP^4\setminus Q$ and we get a non-zero map $f: \Aa _D \to \Aa _O$
for some $D\in \{A,B\}$. Since $\Aa _D$ and $\Aa _O$ are stable, we get that $f$ is an isomorphism. We also get that
the image of $\Aa _P$ in (\ref{eqrm2}) is a direct factor of $\Aa _D$ in $\Ee _{P,O}$. Hence (\ref{eqrm2}) splits, contradicting our assumption.

Now let us deal with the case $r=7$.
From (\ref{eqaa0}) we get the exact sequence
\begin{equation}\label{eqaa1}
0 \to \mathcal {A}^\vee (-1) \to (\mathcal {A}^\vee )^{\oplus 5} \to \mathcal {A}^\vee \otimes \Phi \to 0.
\end{equation}
Note that $h^1(\mathcal {A}^\vee(-1) )=h^1(\mathcal {A}^\vee )=1$ and $h^2(\mathcal {A}^\vee (-1)) = h^1(\mathcal {A}(-2)) =0$. Since
$\Aa$ and $\Phi$ are stable with slopes $1/3$ and $1/4$, respectively, so we have $h^0( \mathcal {A}^\vee \otimes \Phi )=0$ and thus $h^1(\Phi \otimes \mathcal {A}^\vee )=4$.

 Let $\Ee$ be any vector bundle fitting in a non-trivial extension
\begin{equation}\label{eqba1}
0 \to \Phi \to \Ee \stackrel{\psi}{\to}\mathcal {A}_P\to 0.
\end{equation}
Assume that $\Ee$ is decomposable. Since $h^1(\Phi )=0$, $\Ee$ is globally generated. By the list (\ref{dec}) we get $\Ee \cong \Phi \oplus \mathcal {A}_O$ for
some $O\in \PP^4\setminus O$. Composing the inclusion $\Phi \to \Ee$ in (\ref{eqba1}) with the projection of $\Phi \oplus \mathcal {A}_O$
onto its first factor gives a non-zero map $\beta : \Phi \to \Phi $, because $\mathrm{rank} (\Phi )>\mathrm{rank} (\mathcal {A}_O)$. Since $\Phi$ is simple
 by Lemma \ref{ba0}, $\beta$ is an isomorphism. Hence $\beta$ gives a splitting of (\ref{eqba1}), a contradiction.

The case of $r=8$ can be obtained from the proposition \ref{c1}.
\end{proof}

\begin{remark}
There is no indecomposable vector bundle that is an extension of $\Phi$ by $\Aa$. Indeed from the exact sequence
$$\Phi^\vee (-1) \to (\Phi^\vee)^{\oplus 4} \to \Aa \otimes \Phi^{\vee} \to 0,$$
we have $h^0(\Phi^\vee \otimes \Aa)=h^1(\Phi^\vee(-1))=h^0(\Oo_Q)=1$. Taking the dual of (\ref{eqaa0}) and tensoring with $\mathcal {A}$, we get the exact sequence
\begin{equation}\label{eqaa4}
0 \to \Aa \otimes \Phi ^\vee \to  \mathcal {A}^{\oplus 5} \stackrel{\alpha}{\to} \mathcal {A}(1) \to 0.
\end{equation}
We also have $h^1(\mathcal {A}^{\oplus 5})=0$, $h^0(\mathcal {A}^{\oplus 5}) =20$ and $h^0(\mathcal {A}(1)) = h^0(T\PP^3 )+h^0(T\PP^3(-1)) =15+4$.
Using the sequence (\ref{eqaa4}), we have $h^1(\Aa \otimes \Phi^\vee)=0$.

\end{remark}
In order to study spanned vector bundle on $Q$ with $(c_1,c_2,c_3;r)=(2,6,8;8)$ we need the following lemma:
\begin{lemma}\label{l2}The set $\Ww$ of all smooth and connected $C\subset Q$ such that $\deg (C)=6$, $p_a(C)=2$ and $C$ is linearly normal in $\PP^4$ is irreducible. \end{lemma}
\begin{proof}

Fix $C\in \Ww$.
Since $\deg (C)\ge 2p_a(C)+2$
and $C$ is linearly normal, we have $h^1(\PP ^4,\Ii _{C,\PP^4}(2)) =0$
and the homogeneous ideal of $C$ in $\PP^4$ is generated by quadrics \cite{f}. Hence $h^0(\PP^4,\Ii _{C,\PP^4}(2)) = 5$, $h^1(\Ii _C(2))=0$, $h^0(\Ii _C(2)) =4$ and $\Ii _C(2)$ is spanned.
We also know that $\Ii _{C,\PP^4}(2)$ is spanned.

Again let $\Delta$ be the set of all $D\subset Q$ such that $D$ is a smooth complete intersection of $Q$, a quadric hypersurface and a hyperplane. The set $\Delta \subset \mbox{Hilb}(Q)$ is irreducible and of dimension $12$. Let $\Delta ''$ be the set of all nodal curves
$D\cup J \subset Q$ with $D\in \Delta$, $J$ a conic, $\sharp (J\cap D)=2$ such that the line passing through the two points of $D\cap J$ is not contained in $Q$. Each $D\cup J$ is
connected, nodal and
$p_a(D\cup J)=2$.

\quad {\emph {Claim 1:}} Each $D\cup J$ is a smooth point of $\mbox{Hilb}(Q)$ and $\Delta ''$ is irreducible
and of dimension 16.

\quad {\emph {Proof of Claim 1:}} Since $D\cup J$ is nodal, to prove
the first statement it is sufficient to prove $h^1(D\cup J,N_{D\cup J})=0$. We have
$N_D\cong \Oo _D(2)\oplus \Oo _D(1)$, while $N_J\cong \Oo _J(1)\oplus \Oo _J(1)$.
Apply the easy part of \cite{hh}, Theorem 4.1.
Now we prove the second assertion of {\it{Claim 1}}. The set $\Delta$ is irreducible and of dimension 12. Each $D\in \Delta$ is irreducible and hence
the set of all $P, P'\in D$ such that $P\ne P'$ is irreducible and of dimension $2$. Fix $P,P'\in Q$ such that
$P\ne P'$ and the line containing them is not contained in $Q$. The set $D(P,P')$ of all smooth conics $J$ containing $P_1$ and $P_2$
is a non-empty irreducible set of dimension $2$. Now assume $P,P'\in D$. The set of all $J\in D(P,P')$ such
that $J\cap D =\{P,P'\}$ and $J$ is not tangent to $D$ either at $P$ or at $P'$ is a non-empty open subset of $D(P,P')$. \qed

By {\it{Claim 1}} (as in {\it{Claim 3}} of Proposition \ref{b1}) $\Ww$ is irreducible if for every connected component $\Gamma$ of
$\Ww$ we have $\overline{\Gamma}\cap \Delta '' \ne \emptyset$, where $\overline{\Gamma}$ is the closure of $\Gamma$ in $\mbox{Hilb}(Q)$.

\quad{\emph{Claim 2:}} Fix $C\in \Ww$. Let $T\subset Q$ be the intersection of $Q$ with a general quadric hypersurface of $\PP^4$ containing $C$. Then $T$ is a smooth surface.

\quad{\emph{Proof of Claim 2:}} Since $\mathcal {I}_C(2)$ is spanned,
the scheme $T$ is a degree $4$ surface smooth outside $C$ (Bertini's theorem). Hence $T$ is irreducible. Fix $P\in C$. Let $A(C,P)$ denote the set
of all $W\in \vert \Ii _C(2)\vert$ which are singular at $P$. We have $\Ii _C/(\Ii _C)^2 \cong N_C^\vee$. Since $\Ii _C(2)$ is spanned, $N_C^\vee (2)$ is spanned.
Since $N_C^\vee (2)$ has rank $2$ and $N_C^\vee (2)$ is spanned, $A(C,P)$ is a linear subspace with codimension $2$ of $\vert \Ii _C(2)\vert$. Since
$\dim (C)=1$ and $T$ is general, we get $T\notin A(C,P)$ for all $P\in C$. Hence $T$ is smooth. \qed

\begin{proof}[Proof of lemma \ref{l2}]
Fix $T$ as in {\it{Claim 2}}. As in the proof of Proposition \ref{b1}, since $C$ is an integral curve of $T$, but not a line, $\deg (C)=6$ and $p_a(C)=2$, we can represent $C$ by $(a,b_1, \dots, b_5)\in \ZZ^6$ with $a>0$, $a\ge b_1+b_2+b_3$, $b_1\ge b_2\ge b_3\ge b_4\ge b_5\ge 0$,
\begin{equation}\label{eqvaa3}
6 = 3a-b_1-b_2-b_3-b_4-b_5 \ ,
\end{equation}
\begin{equation}\label{eqvaa4}
a^2 = 8+\sum _{i=1}^{5} b_i^2.
\end{equation}
(see \cite{Hartshorne}, V, 4.12 and Ex. 4.8, or \cite{h1}, equation (2) at page 303, for the the case of a cubic surface). Obviously $(4,2,1,1,1,1)$and $(5,2,2,2,2,1)$ are solutions of (\ref{eqvaa3}) and (\ref{eqvaa4}).

 Fix another solution $(a,b_1,b_2,b_3,b_4,b_5)$. Since no plane curve of degree $\le 3$ have normalization of genus
$2$ we may assume $a\ge 4$. We may assume $b_1>0$, because no smooth plane curve has genus 2.
Set $\epsilon := a-b_1-b_2-b_3$. We have $\epsilon \ge 0$ and $2a > b_1+b_2+b_3+b_4+b_5$. Hence it is sufficient to check all cases with $4 \le a \le 5$. First
assume $a=4$. We cannot have $b_1\ge 3$ by (\ref{eqvaa4}). We also cannot have $b_1\le 1$, because they have $b_1+b_2+b_3+b_4+b_5 \le 5$. Hence
$b_1=2$. The genus formula for plane curves gives $b_2\le 1$ and thus $b_2=b_3=b_4=b_5 = 1$. Now assume $a=5$. Since $C$ is not rational,
we have $b_1\le 3$. Since $b_1+b_ 2+b_3+b_4+b_5 =9$, we have $b_1\ge 2$. First assume $b_1=3$. The genus formula for plane curves gives $b_2\le 1$.
Hence $b_1+b_ 2+b_3+b_4+b_5 \le 7$, a contradiction. Now assume $b_1=2$. Since $b_1+b_ 2+b_3+b_4+b_5 =9$, we get $b_2=b_3=b_4=2$ and $b_1=1$.

Thus we may see $C$ as a subcurve of $T$ of either type $(4,2,2,1,1,1)$ or of type $(5,2,2,2,2,1)$.
First
assume $C\in \vert \Oo _T(4,2,1,1,1,1)\vert$. Take a general $D\in \vert \Oo _T(3,1,1,1,1)\vert$
and a general $J\in \vert \Oo _T(1,1,0,0,0,0)\vert$. Notice that $D\in \Delta$ and $J$ is embedded in $T\subset Q\subset \PP^4$
as a smooh conic and $D\cup J\in \Delta ''$. We may deform $C$ to $D\cup J$ inside $T$. Hence we may deform $D\cup J$ inside
$Q$. Now assume $C\in \vert \Oo _T(5,2,2,2,2,1)\vert$. Take a general $D\in \vert \Oo _T(3,1,1,1,1,1) \vert$
and a general $J\in \vert \Oo _T(2,1,1,1,1,0)\vert$. Notice that $D\in \Delta$, $J$ is embedded in $T\subset Q\subset \PP^4$
as a smooh conic and $D\cup J\in \Delta ''$. As in the proof of lemma \ref{l1}, this is sufficient to prove lemma \ref{l2}.
\end{proof}

\begin{proposition}\label{c1}
Every spanned vector bundle on $Q$ with $(c_1,c_2,c_3;r)=(2,6,8;8)$ is isomorphic to $\Phi \oplus \Phi$.
\end{proposition}

Let $\Ww$ be the set of all smooth and connected $C\subset Q$ such that $\deg (C)=6$, $p_a(C)=2$ and $C$ is linearly normal in $\PP^4$. We see $\Ww$ as an open subset of the Hilbert
scheme $\mbox{Hilb}(Q)$ of $Q$. Since
$h^0(C,\omega _C(1)) = 7$, as in Proposition \ref{c1} it is sufficient to prove that for every $C\in \Ww$ we have $\Ee\simeq \Phi\oplus \Phi$ in the exact sequence
$$0\to \Oo _Q^{\oplus 7} \to \Ee \to \Ii _C(2)\to 0.$$
By Theorem 6.3 in \cite{AO} with $(t,j)=(0,1)$ and $\Ff =\Ee^\vee$, it is enough to prove that
$$h^1(\Ee^{\vee}(-2))=h^1(\Ee^{\vee})=h^2(\Ee^{\vee}\otimes \Sigma (-3))=0$$
so that $\Ee$ would have $h^1(\Ee^\vee (-1))=h^1(\Oo_C)=2$ factors of $\Phi$.

Note that $h^1(\Ee^\vee (-2))=h^2(\Ee(-1))=h^2(\Ii_C(1))=h^1(\Oo_C(1))=0$ and $h^1(\Ee^\vee)=h^2(\Ee(-3))$. Since $h^3(\Ee(-3))=h^0(\Ee^\vee)=0$, so we have $h^2(\Ee(-3))=h^2(\Ii_C(-1))-7=h^1(\Oo_C(-1))-7=h^0(\Oo_C(1)\otimes \omega_C)-7=7-7=0$.

So it is sufficient to prove $h^2(\Ee^\vee \otimes \Sigma (-3))=h^1(\Ii_C\otimes \Sigma (1))=0$ for every $C\in \Ww$. Set $\Ww ':= \{C\in \Ww :h^1(\Ii _C\otimes \Sigma (1)) =0\}$.
By the semicontinuity theorem for cohomology, $\Ww '$ is an open subset of $\Ww$. Taking $7$ general sections of $\Phi \oplus \Phi$ we get
$\Ww ' \ne \emptyset$. Take $\Ee$ in the closure of $\Ww '$ in $\Ww$. By the
semicontinuity theorem for cohomology we get non-zero maps $u: \Phi \to \Ee$ and
$v: \Ee \to \Phi$ such that $v\circ u = 0$. Since $\Ee$ is spanned and with no trivial factor, any torsion free quotient of $\Ee$ has
positive $c_1$. Let $\Theta $ be the saturation of $u(\Phi)$. Since $\Ee /\Theta$ has positive $c_1$, we get
$c_1(\Theta )\le 1$. Since $\Phi $ is a stable reflexive sheaf, we first get that $u$ is injective and then get
$\Theta = u(\Phi )$. Let $v': \Ee /u(\Phi ) \to \Phi$ the map induced by $v$. Since $\Phi$ is stable and $c_1(\Phi )=1$, we get
that $v'$ is injective. Since $c_i(\Ee /u(\Phi )) = c_i(\Phi )$ for all $i$, we get that $\Ee$ is an extension of $\Phi$ by $\Phi$. Hence
$\Ee \cong \Phi\oplus \Phi$. Hence by lemma \ref{l2}, we have $\Ww ' = \Ww$, concluding the proof of Proposition \ref{c1}.
\end{proof}

In the case of $c_2=8$, we do not have any possible decomposable vector bundle in the list (\ref{dec}), so every vector bundle in the general extension (\ref{eqa7}) is indecomposable if $r\leq 3+h^1(\Ff^{\vee})$=13. Thus we have the followings :

\begin{proposition}\label{c4}
There exists a globally generated and indecomposable vector bundle $\Ee$ of rank $r\geq 4$ on $Q$ with the Chern classes $(c_1, c_2)=(2,k)$, $k\geq 4$ and $h^0(\Ee(-1))=0$ if and only if we have
\begin{equation}
\left\{
\begin{array}{ll}
 k=5 ~~;~~ 4\leq r \leq 5\\
 k=6 ~~;~~ 4\leq r \leq 7\\
 k=8 ~~;~~ 4\leq r \leq 13.
 \end{array}
 \right.
 \end{equation}
\end{proposition}

\providecommand{\bysame}{\leavevmode\hbox to3em{\hrulefill}\thinspace}
\providecommand{\MR}{\relax\ifhmode\unskip\space\fi MR }
\providecommand{\MRhref}[2]{%
  \href{http://www.ams.org/mathscinet-getitem?mr=#1}{#2}
}
\providecommand{\href}[2]{#2}


\begin{thebibliography}{10}

	
\bibitem{AO}
V. Ancona and G. Ottaviani, \emph{Some applications of {B}eilinson's theorem to projective spaces and quadrics}, Forum Math. \textbf{3} (1991), no. ~2, 157--176. \MR{1092580 (92e:14039)}	
	
\bibitem{am}
C.~Anghel and N.~Manolache, \emph{Globally generated vector bundles on
  $\mathbb{P}^n$ with $c_1=3$}, Preprint, arXiv:1202.6261 [math.AG], 2012.

\bibitem{Arrondo}
E. Arrondo, \emph{A home-made {H}artshorne-{S}erre correspondence}, Rev.
  Mat. Complut. \textbf{20} (2007), no.~2, 423--443. \MR{2351117 (2008g:14084)}

\bibitem{BHM}
E. Ballico, S. Huh,  and F.~Malaspina, \emph{Globally generated vector bundles of
  rank 2 on a smooth quadric threefold}, Preprint, arXiv : 1211.1100 [math.AG], 2012.

\bibitem{BC}
C. B{\u{a}}nic{\u{a}} and J. Coand{\u{a}}, \emph{Existence of rank
  {$3$} vector bundles with given {C}hern classes on homogeneous rational
  {$3$}-folds}, Manuscripta Math. \textbf{51} (1985), no.~1-3, 121--143.
  \MR{788675 (87c:14015)}

\bibitem{Chang}
M.-C. Chang, \emph{A filtered {B}ertini-type theorem}, J. Reine Angew. Math.
  \textbf{397} (1989), 214--219. \MR{993224 (90i:14054)}

 \bibitem{ce}
L. Chiodera and P.~Ellia, \emph{Rank two globally generated vector bundles with
$c_1 \leq 5$}, Rend. Istit. Mat. Univ. Trieste \textbf{44} (2012), 1--10. 

\bibitem{d}
M. Demazure, H.~C. Pinkham, and B. Teissier (eds.),
  \emph{S\'eminaire sur les {S}ingularit\'es des {S}urfaces}, Lecture Notes in
  Mathematics, vol. 777, Springer, Berlin, 1980, Held at the Centre de
  Math{\'e}matiques de l'{\'E}cole Polytechnique, Palaiseau, 1976--1977.
  \MR{579026 (82d:14021)}

 \bibitem{e}
P.~Ellia, \emph{Chern classes of rank two globally generated vector bundles on
  $\mathbb{P}^2$}, Preprint, arXiv:1111.5718 [math.AG], 2011.


\bibitem{f}
T.~Fujita, \emph{Defining equations for certain types of polarized varieties},
  Complex analysis and algebraic geometry, Iwanami Shoten, Tokyo, 1977,
  pp.~165--173. \MR{0437533 (55 \#10457)}

\bibitem{gl}
M.~Green and R.~Lazarsfeld, \emph{Some results on the syzygies of finite sets
  and algebraic curves}, Compositio Math. \textbf{67} (1988), no.~3, 301--314.
  \MR{959214 (90d:14034)}

\bibitem{hh}
R.~Hartshorne and A.~Hirschowitz, \emph{Smoothing algebraic space curves},
  Algebraic geometry, {S}itges ({B}arcelona), 1983, Lecture Notes in Math.,
  vol. 1124, Springer, Berlin, 1985, pp.~98--131. \MR{805332 (87h:14023)}

\bibitem{Hartshorne}
R. Hartshorne, \emph{Algebraic geometry}, Springer-Verlag, New York, 1977,
  Graduate Texts in Mathematics, No. 52. \MR{0463157 (57 \#3116)}

\bibitem{Hartshorne1}
\bysame, \emph{Stable reflexive sheaves}, Math. Ann. \textbf{254} (1980),
  no.~2, 121--176. \MR{597077 (82b:14011)}

\bibitem{h1}
\bysame, \emph{Genre de courbes alg\'ebriques dans l'espace projectif
  (d'apr\`es {L}. {G}ruson et {C}. {P}eskine)}, Bourbaki {S}eminar, {V}ol.
  1981/1982, Ast\'erisque, vol.~92, Soc. Math. France, Paris, 1982,
  pp.~301--313. \MR{689536 (84f:14023)}

  \bibitem{huh}
S. Huh, \emph{On triple {V}eronese embeddings of {$\Bbb P^n$} in the
  {G}rassmannians}, Math. Nachr. \textbf{284} (2011), no.~11-12, 1453--1461.
  \MR{2832657 (2012g:14021)}

\bibitem{m}
N.~Manolache, \emph{Globally generated vector bundles on $\mathbb{P}^3$ with
  $c_1=3$}, Preprint, arXiv:1202.5988 [math.AG], 2012.


\bibitem{mc}
J. McCleary, \emph{User's guide to spectral sequences}, Mathematics Lecture
  Series, vol.~12, Publish or Perish Inc., Wilmington, DE, 1985. \MR{820463
  (87f:55014)}

\bibitem{OSS}
C. Okonek, M. Schneider, and H. Spindler, \emph{Vector bundles
  on complex projective spaces}, Progress in Mathematics, vol.~3, Birkh\"auser
  Boston, Mass., 1980. \MR{561910 (81b:14001)}


\bibitem{OS}
G. Ottaviani and M. Szurek, \emph{On moduli of stable {$2$}-bundles
with small {C}hern classes on {$Q_3$}}, Ann. Mat. Pura Appl. (4) \textbf{167}
(1994), 191--241, With an appendix by Nicolae Manolache. \MR{1313556
(96c:14010)}

\bibitem{sierra}
J.-C. Sierra, \emph{A degree bound for globally generated vector
  bundles}, Math. Z. \textbf{262} (2009), no.~3, 517--525. \MR{2506304
  (2010i:14076)}

\bibitem{SU}
J.-C. Sierra and L. Ugaglia, \emph{On globally generated vector
  bundles on projective spaces}, J. Pure Appl. Algebra \textbf{213} (2009),
  no.~11, 2141--2146. \MR{2533312 (2010d:14062)}

  \bibitem{SU2}
J.-C. Sierra and L. Ugaglia, \emph{On globally generated vector
  bundles on projective spaces ii}, Preprint, arXiv:1203.0185 [math.AG], 2012.


\end{thebibliography}
\end{document}